\documentclass[11pt,twoside]{article}

\usepackage{fullpage}
\usepackage{local_smooth}
\usepackage{amsfonts}
\usepackage{amssymb}
\usepackage{amsfonts}
\usepackage{amssymb}
\usepackage{multirow}

\makeatletter
\long\def\@makecaption#1#2{
  \vskip 0.8ex
  \setbox\@tempboxa\hbox{\small {\bf #1:} #2}
  \parindent 1.5em  
  \dimen0=\hsize
  \advance\dimen0 by -3em
  \ifdim \wd\@tempboxa >\dimen0
  \hbox to \hsize{
    \parindent 0em
    \hfil 
    \parbox{\dimen0}{\def\baselinestretch{0.96}\small
      {\bf #1.} #2
    } 
    \hfil}
  \else \hbox to \hsize{\hfil \box\@tempboxa \hfil}
  \fi
}
\makeatother

\newcommand{\statprob}{\ensuremath{P}}

\newcommand{\defn}{\ensuremath{: \, = }}
\newcommand{\subexpbound}{\ensuremath{\Lambda}}
\newcommand{\subexpparam}{\ensuremath{\tau}}
\newcommand{\subexpprob}{\ensuremath{a}}
\newcommand{\subexpexp}{\ensuremath{\alpha}}

\usepackage{xifthen}

\newcommand{\smoothedfunc}[1][]{%
  \ifthenelse{\isempty{#1}}{%
    f_\smoothingdist
  }{%
    f_{\smoothingdist_{#1}}
  }
}

\long\def\comment#1{}

\begin{document}

\begin{center}
  {\bf{\LARGE { Randomized Smoothing for Stochastic Optimization }}}

\vspace*{.15in}

\begin{tabular}{ccc}
  John C. Duchi$^1$ & Peter Bartlett$^{1,2}$ & Martin J. Wainwright$^{1,2}$ \\
  jduchi@eecs.berkeley.edu & bartlett@eecs.berkeley.edu &
  wainwrig@eecs.berkeley.edu
\end{tabular}

\vspace*{.15in}

\begin{tabular}{cc}
  Department of Electrical Engineering and Computer Sciences$^1$ &
  Department of Statistics$^2$
\end{tabular} \\
University of California Berkeley, Berkeley, CA 94720
\vspace*{.15in}

March 2012

\begin{abstract}
  We analyze convergence rates of stochastic optimization procedures
  for non-smooth convex optimization problems.  By combining
  randomized smoothing techniques with accelerated gradient methods,
  we obtain convergence rates of stochastic
  optimization procedures, both in expectation and with high
  probability, that have optimal dependence on the variance of
  the gradient estimates. To the best of our knowledge, these are the first
  variance-based rates for non-smooth optimization. We give several
  applications of our results to statistical estimation problems, and
  provide experimental results that demonstrate the effectiveness of
  the proposed algorithms.  We also describe how a combination of our
  algorithm with recent work on decentralized optimization yields a
  distributed stochastic optimization algorithm that is order-optimal.
\end{abstract}

\end{center}

\section{Introduction}


In this paper, we develop and analyze randomized smoothing procedures for
solving the following class of stochastic optimization problems.  Let $\{
F(\cdot \, ; \statsample), \, \statsample \in \statsamplespace \}$ be a
collection of real-valued functions, each with domain containing the closed
convex set $\xdomain \subseteq \real^d$.  Letting $\statprob$ be a probability
distribution over the index set $\statsamplespace$, consider the function $f:
\xdomain \rightarrow \real$ defined via
\begin{align}
  \label{eqn:def-F}
  f(x) & \defn \E \big[F(x; \statsample) \big] \; = \;
  \int_\statsamplespace F(x; \statsample) d \statprob(\statsample).
\end{align}
In this paper, we analyze a family of randomized smoothing procedures
for solving potentially non-smooth stochastic optimization problems of
the form
\begin{equation}
  \label{eqn:objective}
  \min_{x \in \xdomain} ~ \big \{ f(x) + \regularizer(x) \big \},
\end{equation}
where $\regularizer: \xdomain \rightarrow \real$ is a known
regularizing function.  Throughout the paper, we assume that $f$ is
convex on its domain $\xdomain$.  This condition is satisfied, for
instance, if the function $F(\cdot; \statsample)$ is convex for
$\statprob$-almost every $\statsample$.  We assume that $\regularizer$
is closed and convex, but we allow for non-differentiability so that
the framework includes the $\ell_1$-norm and related regularizers.

While we will later discuss the effects that $\regularizer(x)$ has
on our optimization procedures, throughout we will mostly consider the
properties of the stochastic function
$f$. Problem~\eqref{eqn:objective} is challenging mainly for two
reasons. First, the function $f$ may be non-smooth. Second, in many
cases, $f$ cannot actually be evaluated. When $\statsample$ is
high-dimensional, the integral~\eqref{eqn:def-F} cannot be efficiently
computed, and in statistical learning problems we usually do not even
know what the distribution $\statprob$ is. Thus, throughout this work,
we assume only that we have access to a stochastic oracle that allows
us to get i.i.d.\ samples $\statsample \sim \statprob$, and
consequently we focus on stochastic gradient procedures for the convex
program~\eqref{eqn:objective}.

To address the first difficulty mentioned above---namely that $f$ may
be non-smooth---several researchers have considered techniques for
smoothing the objective.  Such approaches for deterministic non-smooth
problems are by now well-known, and include Moreau-Yosida
regularization (e.g.~\cite{LemarechalSa97}),
methods based on recession functions~\cite{Ben-TalTe89}; and a method
that uses conjugate and proximal functions~\cite{Nesterov05b}.
Several works study methods to replace constraints $f(x) \le 0$ in
convex programming problems with exact penalties $\max\{0, f(x)\}$ in
the objective, after which smoothing is applied to the $\max\{0,
\cdot\}$ operator (e.g., see the paper~\cite{ChenMa96} and references
therein). The difficulty of such approaches is that most require quite
detailed knowledge of the structure of the function $f$ to be
minimized and hence are impractical in stochastic settings.

The second difficulty of solving the convex
program~\eqref{eqn:objective} is that the function cannot actually be
evaluated except through stochastic realizations of $f$ and its
(sub)gradients. In this paper, we develop an algorithm for solving
problem~\eqref{eqn:objective} based on stochastic subgradient methods.
Although such methods are
classical~\cite{RobbinsMo51,Ermoliev69,PolyakJu92}, recent work by
Juditsky et al.~\cite{JuditskyNeTa08} and Lan~\cite{Lan09thesis,Lan10}
has shown that if $f$ is smooth---its gradients are
Lipschitz continuous---convergence rates dependent on the variance of
the stochastic gradient estimator are achievable. Specifically, if
$\stddev^2$ is the variance of the gradient estimator, the convergence
rate of the resulting stochastic optimization procedure is
$\order(\stddev / \sqrt{T})$. Of particular relevance to our study is
the following fact: if the oracle (instead of returning just a single
estimate) returns $m$ unbiased estimates of the gradient, the variance
of the gradient estimator is reduced by a factor of $m$. Dekel
et al.~\cite{DekelGiShXi12} exploit this fact to develop asymptotically
order-optimal distributed optimization algorithms, as we discuss in
the sequel.

To the best of our knowledge, there is no work on \emph{non-smooth}
stochastic problems for which a reduction in the variance of the
stochastic estimate of the true subgradient gives an improvement in
convergence rates. For non-smooth stochastic optimization, known
convergence rates are dependent only on the Lipschitz constant of the
functions $F(\cdot; \statsample)$ and the number of actual updates
performed. Within the oracle model of convex
optimization~\cite{NemirovskiYu83}, the optimizer has access to a
black-box oracle that, given a point $x \in \xdomain$, returns an
unbiased estimate of a (sub)gradient of the objective $f$ at the point
$x$. In most stochastic optimization procedures, an algorithm updates
a parameter $x_t$ at every iteration by querying the oracle for one
stochastic subgradient; we consider the natural extension to the case
when the optimizer issues several queries to the stochastic oracle at
every iteration.

A convolution-based smoothing technique amenable to non-smooth
stochastic optimization problems is the starting point for our
approach.  A number of authors
(e.g.,~\cite{KatkovnikKu72,Rubinstein81,LakshmananFa08,YousefianNeSh12})
have noted that particular random perturbations of the variable $x$
transform $f$ into a smooth function.  The intuition underlying such
approaches is that convolving two functions yields a new function that
is at least as smooth as the smoothest of the two original
functions. In particular, let $\smoothingdist$ denote the density of a
random variable with respect to Lebesgue measure, and consider the
smoothed objective function
\begin{equation}
  \label{eqn:convolve}
  \smoothedfunc(x)
  \defeq \int_{\R^d} f(x + y) \smoothingdist(y) dy
  = \E_\smoothingdist[f(x + Z)],
\end{equation}
where $Z$ is a random variable with probability density
$\smoothingdist$. Clearly, $\smoothedfunc$ is convex whenever $f$
is convex; moreover, it is known that if $\smoothingdist$ is a density
with respect to Lebesgue measure, then $\smoothedfunc$ is
differentiable~\cite{Bertsekas73}.

We analyze minimization procedures that solve the non-smooth
problem~\eqref{eqn:objective} by using stochastic gradient samples from the
smoothed function~\eqref{eqn:convolve} with appropriate choice of
smoothing density $\smoothingdist$.  The main contribution of our
paper is to show that the ability to issue several queries to the
stochastic oracle for the original objective~\eqref{eqn:objective} can
give faster rates of convergence than a simple stochastic oracle. Our
two main theorems quantify the above statement in terms of expected
values (Theorem~\ref{theorem:main-theorem-no-horizon}) and, under an
additional reasonable tail condition, with high probability
(Theorem~\ref{theorem:main-theorem-highprob-no-horizon}). One
consequence of our results is that a procedure that queries the
non-smooth stochastic oracle for $m$ subgradients at iteration $t$
achieves rate of convergence $\order(\radius \lipobj / \sqrt{Tm})$ in
expectation and with high probability.  (Here $\lipobj$ is the
Lipschitz constant of the function $f$ and $\radius$ is the
$\ell_2$-radius of its domain.)  As we discuss in
Section~\ref{sec:corollaries}, this convergence rate is optimal up to
constant factors.  Moreover, this fast rate of convergence has
implications for applications in statistical problems, distributed
optimization, and other areas, as discussed in
Section~\ref{sec:applications}.

The remainder of the paper is organized as follows.  In the next section, we
review standard techniques for stochastic optimization, noting a few
of their deficiencies. After this, we state our algorithm and main theorems
achieving faster rates of convergence for non-smooth stochastic problems using
the randomized smoothing technique~\eqref{eqn:convolve}. We make strong use of
the fine analytic properties of randomized smoothing, and collect several
relevant results in Appendix~\ref{appendix:smoothing}. In
Section~\ref{subsec:applications}, we outline several applications of
the smoothing techniques, which we complement in Section~\ref{sec:experiments}
with experiments and simulations showing the merits of our new
approach. Section~\ref{sec:main-result-proofs} contains proofs of our main
results, though we defer more technical aspects to the appendices.

\paragraph{Notation:}

For the reader's convenience, here we specify notation as well as a few
definitions.  We use $B_p(x, u) = \{y \in \R^d \mid \norm{x - y}_p \le u\}$ to
denote the closed $p$-norm ball of radius $u$ around the point $x$. Addition
of sets $A$ and $B$ is defined as the Minkowski sum in $\R^d$, that is, $A + B
= \{x \in \R^d \mid x = y + z, y \in A, z \in B\}$, and multiplication of a
set $A$ by a scalar $\alpha$ is defined to be $\alpha A \defeq \{\alpha x \mid
x \in A\}$. For any function or distribution $\smoothingdist$, we let $\supp
\smoothingdist \defeq \{x \mid f(x) \neq 0\}$ denote its support.  Given a
convex function $f$ with domain $\xdomain$, for any $x \in \xdomain$, we use
$\partial f(x)$ to denote its subdifferential. We define the shorthand
notation $\norm{\partial f(x)} = \sup \{ \norm{g} \, \mid \, g \in \partial
f(x) \}$ for any norm $\norm{\cdot}$.  The dual norm $\dnorm{\cdot}$ with the
norm $\norm{\cdot}$ is defined as $\dnorm{z} \defeq \sup_{\norm{x} \le 1} \<z,
x\>$.  A function $f$ is $\lipobj$-Lipschitz with respect to the norm
$\norm{\cdot}$ over $\xdomain$ if
\begin{equation*}
  |f(x) - f(y)| \le \lipobj \norm{x - y}
\end{equation*}
for all $x, y \in \xdomain$. For convex $f$, it is
known~\cite{HiriartUrrutyLe96} that $f$ is $\lipobj$-Lipschitz in this
sense if and only if $\sup_{x \in \xdomain} \dnorm{\partial f(x)} \leq
\lipobj$.
We say the gradient of $f$ is $\lipgrad$-Lipschitz continuous with
respect to the norm $\norm{\cdot}$ over $\xdomain$ if
\begin{equation*}
  \dnorm{\nabla f(x) - \nabla f(y)} \le \lipgrad \norm{x - y}
  ~~~ \mbox{for} ~
  x, y \in \xdomain.
\end{equation*}
A function $\prox$ is strongly convex with respect to a norm $\norm{\cdot}$
over $\xdomain$ if for all $x, y, \in \xdomain$,
\begin{equation*}
  \prox(y) \ge \prox(x) + \<\nabla \prox(x), y - x\> + \half \norm{x - y}^2.
\end{equation*}
Given a convex and differentiable function $\prox$, the associated Bregman
divergence~\cite{Bregman67} is given by \mbox{ $\divergence(x, y) \defeq
  \prox(x) - \prox(y) - \<\nabla\prox(y), x - y\>$.}  When $X \in \R^{d_1
  \times d_2}$ is a matrix, we let $\singval_i(X)$ denote its $i$th largest
singular value, and when $X \in \R^{d \times d}$, we let $\lambda_i(X)$ denote
its $i$th largest eigenvalue by modulus. The transpose of $X$ is denoted
$X^\top$. The notation $\statsample \sim \statprob$ indicates that
$\statsample$ is drawn according to the distribution $\statprob.$

\section{Main results and some consequences}
\label{sec:main-results}

In this section, we begin by motivating the algorithm studied in this
paper, and then state our main results on its convergence behavior.

\subsection{Some background}

We focus on stochastic gradient descent methods\footnote{We note in passing
  that essentially identical results can also be obtained for methods based on
  mirror descent~\cite{NemirovskiYu83,Tseng08}, though we omit these so as not
  to overburden the reader.}  based on dual averaging
schemes~\cite{Nesterov09} for solving the stochastic
problem~\eqref{eqn:objective}.  Dual averaging methods are based on a proximal
function $\prox$, which is assumed strongly convex with respect to a norm
$\norm{\cdot}$. The update scheme of such a method is as follows.  Given a
point $x_t \in \xdomain$, the algorithm queries a stochastic oracle and
receives a random vector $g_t$ such that $\E[g_t] \in \partial f(x_t)$.  The
algorithm then performs the update
\begin{equation}
  \label{eqn:dual-averaging-update}
  x_{t+1} = \argmin_{x \in \xdomain} \bigg\{\sum_{\tau=0}^t \<g_\tau, x\> +
  \frac{1}{\stepsize_t} \prox(x)\bigg\}
\end{equation}
where $\stepsize_t > 0$ is a sequence of stepsizes.
Under some mild assumptions, the algorithm is guaranteed to converge for
stochastic problems. 
%
%
%
For instance, suppose that $\prox$ is strongly convex with respect to the norm
$\norm{\cdot}$, and moreover that $\E[\dnorm{g_t}^2] \le \lipobj^2$ for all
$t$, where we recall that $\dnorm{\cdot}$ denotes the dual norm to
$\norm{\cdot}$. Then, with stepsizes $\stepsize_t \propto \radius / \lipobj
\sqrt{t}$, it is known that the sequence $\{x_t\}_{t=0}^\infty$ generated by
the updates~\eqref{eqn:dual-averaging-update} satisfies
\begin{align}
\label{EqnOldRate}
    \E \left[f \left(\frac{1}{T} \sum_{t=1}^T x_t \right) \right] - f(x^*) =
    \order \left(\frac{\lipobj \sqrt{\prox(x^*)}}{\sqrt{T}} \right).
\end{align}
We refer the reader to papers by Nesterov~\cite{Nesterov09} and
Xiao~\cite{Xiao10} for results of this type.

An unsatisfying aspect of the bound~\eqref{EqnOldRate} is the absence
of any role for the variance of the (sub)gradient estimator $g_t$. In
particular, even if an algorithm is able to obtain $m > 1$ samples of
the gradient of $f$ at $x_t$---thereby giving a significantly more
accurate gradient estimate---this result fails to capture the likely
improvement of the method.  We address this problem by stochastically
smoothing the non-smooth objective $f$ and then adapt recent work on
so-called ``accelerated'' gradient methods~\cite{Lan10,Tseng08,Xiao10}
to achieve variance-based improvements.
Accelerated methods work only when the function $f$ is smooth---that is, when
it has Lipschitz continuous gradients.  Thus, we turn now to developing the
tools necessary to stochastically smooth the non-smooth
objective~\eqref{eqn:objective}.


\subsection{Description of algorithm}

Our algorithm is based on observations of stochastically perturbed gradient
information at each iteration, where we slowly decrease the perturbation as
the algorithm proceeds. More precisely, our algorithm uses the following
scheme. Let $\{\smoothparam_t\} \subset \R_+$ be a non-increasing sequence of
positive real numbers; these quantities control the perturbation size. At
iteration $t$, rather than query the stochastic oracle at the point $y_t$, the
algorithm queries the oracle at $m$ points drawn randomly from some
neighborhood around $y_t$.  Specifically, it performs the following three
steps:
\begin{enumerate}[(1)]
\item Draws random variables $\{Z_{i, t}\}_{i=1}^m$ in an
  i.i.d.\ manner according to the distribution $\smoothingdist$.
\item Queries the oracle at the $m$ points $y_t + \smoothparam_t
  Z_{i,t}, i = 1, 2, \ldots, m$, yielding stochastic gradients
  \begin{align}
    \label{eqn:average-smooth-gradient}
    g_{i,t} & \in \partial F(y_t + \smoothparam_t Z_{i,t}, \statsample_{i,t}),
    \quad \mbox{where $\statsample_{i,t} \sim \statprob$, for $i = 1, 2,
      \ldots, m$.} 
  \end{align}
\item Computes the average $g_t = \frac{1}{m} \sum_{i=1}^m g_{i,t}$.
\end{enumerate}
Here and throughout we denote the distribution of the random variable
$\smoothparam_t Z$ by $\smoothingdist_t$, and we note that this procedure
ensures \mbox{$\E[g_t \mid y_t] = \nabla\smoothedfunc[t](y_t) = \nabla
  \E[F(y_t + \smoothparam_t Z; \statsample) \mid y_t]$,} where
$\smoothedfunc[t]$ is the smoothed function~\eqref{eqn:convolve} and
$\smoothingdist_t$ is the density of $\smoothparam_t$.

By combining the sampling scheme~\eqref{eqn:average-smooth-gradient} with
extensions of Tseng's recent work on accelerated gradient
methods~\cite{Tseng08}, we can achieve stronger convergence rates for solving
the non-smooth objective~\eqref{eqn:objective}. The update we propose is
essentially a smoothed version of the simpler
method~\eqref{eqn:dual-averaging-update}. The method uses three series of
points, denoted $\{x_t, y_t, z_t\} \in \xdomain^3$. We use $y_t$ as a ``query
point'', so that at iteration $t$, the algorithm receives a vector $g_t$ as
described in the sampling scheme~\eqref{eqn:average-smooth-gradient}.  The
three sequences evolve according to a dual-averaging
algorithm, which in our case involves three scalars $(L_t, \theta_t,
\extrastep_t)$ to control step sizes.  The recursions are as follows:
\begin{subequations}
  \begin{align}
    y_t & = (1 - \theta_t) x_t + \theta_t z_t \label{eqn:y-update} \\
    z_{t+1} & = \argmin_{x \in \xdomain}
    \bigg\{\sum_{\tau=0}^t \frac{1}{\theta_\tau} \<g_\tau, x\> +
    \sum_{\tau=0}^t \frac{1}{\theta_\tau} \regularizer(x)
    + L_{t+1} \prox(x) +
    \frac{\extrastep_{t+1}}{\theta_{t+1}} \prox(x) \bigg\}
    \label{eqn:z-update} \\
    x_{t+1} & = (1 - \theta_t) x_t +  \theta_t z_{t+1}.
    \label{eqn:x-update}
  \end{align}
\end{subequations}
In prior work on accelerated schemes for stochastic and non-stochastic
optimization~\cite{Tseng08,Lan10,Xiao10}, the term $L_t$ is set equal
to the Lipschitz constant of $\nabla f$; in contrast, our choice of
varying $L_t$ allows our smoothing schemes to be oblivious to the
number of iterations $T$. The extra damping term $\extrastep_t /
\theta_t$ provides control over the fluctuations induced by using the
random vector $g_t$ as opposed to deterministic subgradient
information. As in Tseng's work~\cite{Tseng08}, we assume that $\theta_0
= 1$ and \mbox{$(1 - \theta_t) / \theta_t^2 = 1 / \theta_{t-1}^2$;}
the latter equality is ensured by setting $\theta_t = 2 / (1 + \sqrt{1
  + 4 / \theta_{t-1}^2})$.

\subsection{Convergence rates}

We now state our two main results on the convergence rate of the
randomized smoothing procedure~\eqref{eqn:average-smooth-gradient}
with accelerated dual averaging
updates~\eqref{eqn:y-update}--\eqref{eqn:x-update}.  So as to avoid
cluttering the theorem statements, we begin by stating our main
assumptions and notation.  When we state that a function $f$ is
Lipschitz continuous, we mean with respect to the norm $\norm{\cdot}$,
whose dual norm we denote $\dnorm{\cdot}$, and we assume that $\prox$
is nonnegative and strongly convex with respect to $\norm{\cdot}$.  Our main
assumption ensures that the smoothing operator and smoothed function
$\smoothedfunc$ are relatively well-behaved.
\begin{assumption}[Smoothing properties]
  \label{assumption:smoothness}
  The random variable $Z$ is zero-mean with density $\smoothingdist$ (with
  respect to Lebesgue measure on the affine hull $\aff(\xdomain)$ of
  $\xdomain$), and there are constants $\lipobj$ and $\lipgrad$ such that for
  all $u > 0$, $\E[f(x + uZ)] \le f(x) + \lipobj u$, and $\E[f(x + uZ)]$ has
  $\frac{\lipgrad}{u}$-Lipschitz continuous gradient with respect to the norm
  $\norm{\cdot}$.
  For $\statprob$-almost every $\statsample \in \statsamplespace$, we have
  $\dom F(\cdot; \statsample) \supseteq \smoothparam_0 \supp
  \smoothingdist + \xdomain$.
\end{assumption}
\noindent
Let $\smoothingdist_t$ denote the density of the random vector $\smoothparam_t
Z$ and define the instantaneous smoothed function $\smoothedfunc[t] = \int
f(x + z) d\smoothingdist_t(z)$.
As discussed in the introduction, the function
$\smoothedfunc[t]$ is guaranteed to be smooth whenever $\smoothingdist$ (and
hence $\smoothingdist_t$) is a density with respect to Lebesgue measure, so
Assumption~\ref{assumption:smoothness} ensures that $\smoothedfunc[t]$ is
uniformly close to $f$ and not too ``jagged.''  Many smoothing distributions,
including Gaussians and uniform distributions on norm balls, satisfy
Assumption~\ref{assumption:smoothness} (see
Appendix~\ref{appendix:smoothing}); we use such examples in the corollaries to
follow.  The containment of $\smoothparam_0 \supp \smoothingdist + \xdomain$
in $\dom F(\cdot; \statsample)$ guarantees that the subdifferential $\partial
F(\cdot; \statsample)$ is non-empty at all sampled points $y_t +
\smoothparam_t Z$. Indeed, since $\smoothingdist$ is a density with respect to
Lebesgue measure on $\aff(\xdomain)$, with probability one $y_t +
\smoothparam_t Z \in \relint \dom F(\cdot; \statsample)$ and
thus~\cite{HiriartUrrutyLe96} the subdifferential $\partial F(y_t +
\smoothparam_t Z; \statsample) \neq \emptyset$.  There are many smoothing
distributions $\smoothingdist$, including standard Gaussian and uniform
distributions on norm balls, for which Assumption~\ref{assumption:smoothness}
holds (see Appendix~\ref{appendix:smoothing}), and we use such examples in the
corollaries to follow. \\

In the algorithm~\eqref{eqn:y-update}--\eqref{eqn:x-update}, we set $L_t$ to
be an upper bound on the Lipschitz constant $\frac{\lipgrad}{\smoothparam_t}$
of the gradient of $\E[f(x + \smoothparam_t Z)]$; this choice ensures good
convergence properties of the algorithm. The following is the first of our
main theorems.
\begin{theorem}
  \label{theorem:main-theorem-no-horizon}
  Define $\smoothparam_t = \theta_t \smoothparam$, use the scalar
  sequence $L_t = \lipgrad / \smoothparam_t$, and assume
  that $\extrastep_t$ is non-decreasing. Under
  Assumption~\ref{assumption:smoothness}, for any $x^* \in \xdomain$
  and $T \ge 4$,
  \begin{equation}
    \label{EqnNoHorizon}
    \E[f(x_T) + \regularizer(x_T)] - [f(x^*) + \regularizer(x^*)] \le
    \frac{6 \lipgrad \prox(x^*)}{T \smoothparam} + \frac{2
      \extrastep_T \prox(x^*)}{T} + \frac{1}{T} \sum_{t=0}^{T-1}
    \frac{1}{\extrastep_t} \E\big[\dnorm{\error_t}^2\big] + \frac{4
      \lipobj \smoothparam}{T},
  \end{equation}
  where $\error_t \defn \nabla \smoothedfunc[t](y_t) - g_t$ is the
  error in the gradient estimate.
\end{theorem}

\paragraph{Remarks:}  
Note that the convergence rate~\eqref{EqnNoHorizon} involves the
variance $\E[\dnorm{\error_t}^2]$ explicitly.  We exploit this fact in
the corollaries to be stated shortly.  In addition, note that
Theorem~\ref{theorem:main-theorem-no-horizon} does not require a
priori knowledge of the number of iterations $T$ to be performed,
which renders it suitable to online and streaming applications.  If
such knowledge is available, then it is possible to give a similar
result using the smoothing parameter $\smoothparam_t \equiv u$ for all
$t$; such a result is stated as Theorem~\ref{theorem:main-theorem} in
Section~\ref{sec:main-result-proofs}. \\


The preceding result, which provides convergence in expectation, can
be extended to bounds that hold with high probability under suitable
tail conditions on the error $\error_t \defn \nabla \smoothedfunc[t](y_t) -
g_t$.  In particular, let $\mc{F}_t$ denote the $\sigma$-field of the
random variables $g_{i,s}$, $i = 1, \ldots, m$ and $s = 0, \ldots, t$.
In order to achieve high-probability convergence results, a subset of
our results involve the following assumption.
\begin{assumption}[Sub-Gausian errors]
  \label{assumption:subexp-errors}
  The error is \emph{$(\dnorm{\cdot}, \stddev)$ sub-Gaussian}
  for some $\stddev > 0$,
  meaning that with probability 1
  \begin{align}
    \label{assumption:mgf-squared}
    \E[\exp(\dnorm{\error_t}^2/\stddev^2) \mid \mc{F}_{t-1}] & \le \exp(1)
    \quad \mbox{for all $t \in \{1, 2, 3, \ldots \}$.}
  \end{align}
\end{assumption}
\noindent
We refer the reader to Appendix~\ref{AppTail} for more background on
sub-Gaussian and sub-exponential random variables.  In past work on
smooth optimization, other authors~\cite{JuditskyNeTa08,Lan10,Xiao10}
have imposed this type of tail assumption, and we discuss sufficient
conditions for the assumption to hold in
Corollary~\ref{corollary:highprob} in the following section.

\begin{theorem}
  \label{theorem:main-theorem-highprob-no-horizon}
 In addition to the conditions of
 Theorem~\ref{theorem:main-theorem-no-horizon}, assume $\xdomain$ is
 compact with \mbox{$\norm{x - x^*} \le \radius$} for all $x \in \xdomain$
 and that Assumption~\ref{assumption:subexp-errors} holds.  Then with
 probability at least $1 - 2 \delta$, the algorithm with step size
 $\extrastep_t = \extrastep \sqrt{t + 1}$ satisfies
  \begin{align*}
    f(x_T) + \regularizer(x_T) -
    [f(x^*) + \regularizer(x^*)] & \le \frac{6 \lipgrad
      \prox(x^*)}{T \smoothparam} + \frac{4 \lipobj
      \smoothparam}{T} + \frac{4 \extrastep_T \prox(x^*)}{T + 1}
    + \theta_{T-1} \sum_{t=0}^{T-1} \frac{1}{2\extrastep_t}
    \E[\dnorm{\error_t}^2] \\ & \qquad ~ + \frac{4 \stddev^2
      \max\left\{\log\frac{1}{\delta}, \sqrt{2 e (\log T + 1)
        \log\frac{1}{\delta}}\right\}}{ \extrastep T} +
    \frac{\stddev \radius \sqrt{\log\frac{1}{\delta}}}{\sqrt{T}}.
  \end{align*}
\end{theorem}

\paragraph{Remarks:} 
The first four terms in the convergence rate
Theorem~\ref{theorem:main-theorem-highprob-no-horizon} gives are essentially
identical to the expected rate of
Theorem~\ref{theorem:main-theorem-no-horizon}. The first of the additional
terms decreases at a rate of $1 / T$, while the second decreases at a rate of
$\stddev / \sqrt{T}$. As we discuss in the Corollaries that follow, the
dependence $\stddev / \sqrt{T}$ on the variance $\stddev^2$ is optimal, and an
appropriate choice of the sequence $\extrastep_t$ in
Theorem~\ref{theorem:main-theorem-no-horizon} yields the same rates to
constant factors.

\subsection{Some consequences}
\label{sec:corollaries}

The corollaries of the above theorems---and the consequential optimality
guarantees of the algorithm above---are our main focus for the remainder of
this section. Specifically, we show concrete convergence bounds for algorithms
using different choices of the smoothing distribution $\smoothingdist$. For
each corollary, we make the assumption that $x^* \in \xdomain$ satisfies
$\prox(x^*) \le \radius^2$, but is otherwise arbitrary, that the
iteration number $T \ge 4$, and that $\smoothparam_t = \smoothparam \theta_t$.

We begin with a corollary that provides bounds when the smoothing
distribution $\smoothingdist$ is uniform on the $\ell_2$-ball.  The
conditions on $F$ in the corollary hold, for example, when $F(\cdot;
\statsample)$ is $\lipobj$-Lipschitz with respect to the $\ell_2$-norm
for $\statprob$-a.e.\ sample of $\statsample$.
\begin{corollary}
  \label{corollary:variance-rate-ltwo}
  Let $\smoothingdist$ be uniform on $B_2(0, 1)$ and assume $\E[\ltwo{\partial
      F(x; \statsample)}^2] \le \lipobj^2$ for $x \in \xdomain + B_2(0,
  \smoothparam)$, where we set $u = \radius d^{1/4}$. With the
  step size choices $\extrastep_t =
  \lipobj \sqrt{t + 1} / \radius \sqrt{m}$ and
  $L_t = \lipobj \sqrt{d} / \smoothparam_t$,
  \begin{equation*}
    \E[f(x_T) + \regularizer(x_T)] - [f(x^*) + \regularizer(x^*)]
    \le \frac{10 \lipobj \radius d^{1/4}}{T}
    + \frac{5 \lipobj \radius}{\sqrt{Tm}}.
  \end{equation*}
\end{corollary}
\noindent 
The following corollary shows that similar convergence rates are attained when
smoothing with the normal distribution.
\begin{corollary}
  \label{corollary:variance-rate-normal}
  Let $\smoothingdist$ be the $d$-dimensional normal distribution with
  zero-mean and identity covariance $I$ and assume that $F(\cdot;
  \statsample)$ is $\lipobj$-Lipschitz with respect to the $\ell_2$-norm for
  $\statprob$-a.e.\ $\statsample$.  With smoothing parameter
  $\smoothparam = \radius d^{-1/4}$ and step sizes
  $\extrastep_t = \lipobj \sqrt{t + 1} / \radius \sqrt{m}$ and
  $L_t = \lipobj / \smoothparam_t$, we have
  \begin{equation*}
    \E[f(x_T) + \regularizer(x_T)] - [f(x^*) + \regularizer(x^*)]
    \le \frac{10 \lipobj\radius d^{1/4}}{T}
    + \frac{5 \lipobj\radius}{\sqrt{Tm}}.
  \end{equation*}
\end{corollary}
\noindent
We remark here (deferring deeper discussion to
Lemma~\ref{lemma:smoothing-sharp}) that the dimension dependence of $d^{1/4}$
on the $1/T$ term in the previous corollaries cannot be
improved by more than a constant factor. Essentially, functions $f$ exist
whose smoothed version $\smoothedfunc$ cannot have both Lipschitz
continuous gradient and be uniformly close to $f$ in a dimension-independent
sense, at least for the uniform or normal distributions.

The advantage of using normal random variables---as opposed to $Z$
uniform on $B_2(0, u)$---is that no normalization of $Z$ is necessary,
though there are more stringent requirements on $f$. The lack of
normalization is a useful property in very high dimensional scenarios,
such as statistical natural language
processing~\cite{ManningRaSc08}. Similarly, we can sample $Z$ from an
$\ell_\infty$ ball, which, like $B_2(0, u)$, is still compact, but
gives slightly looser bounds than sampling from $B_2(0,
u)$. Nonetheless, it is much easier to sample from $B_\infty(0, u)$ in
high dimensional settings, especially sparse data scenarios such as
NLP where only a few coordinates of the random variable $Z$ are
needed.

There are several objectives $f + \regularizer$ with domains $\xdomain$ for
which the natural geometry is non-Euclidean, which motivates the mirror
descent family of algorithms~\cite{NemirovskiYu83}. By using different
distributions $\smoothingdist$ for the random perturbations $Z_{i,t}$ in
\eqref{eqn:average-smooth-gradient}, we can take advantage of non-Euclidean
geometry. Here we give an example that is quite useful for problems in which
the optimizer $x^*$ is sparse; for example, the optimization set $\xdomain$
may be a simplex or $\ell_1$-ball, or $\regularizer(x) = \lambda \lone{x}$.
The idea in this corollary is that we achieve a pair of dual norms that may
give better optimization performance than the $\ell_2$-$\ell_2$ pair above.
\begin{corollary}
  \label{corollary:variance-rate-linf}
  Let $\smoothingdist$ be the uniform density on $B_\infty(0, 1)$
  and assume that $F(\cdot; \statsample)$ is $\lipobj$-Lipschitz continuous
  with respect to the $\ell_1$-norm over $\xdomain + B_\infty(0,
  \smoothparam)$ for $\statsample \in \statsamplespace$, where
  we set $\smoothparam = \radius \sqrt{d \log d}$. Use the proximal
  function $\prox(x) = \frac{1}{2(p - 1)} \norm{x}_p^2$ for $p = 1 + 1/\log
  d$ and set $\extrastep_t =
  \lipobj \sqrt{t + 1} / \radius \sqrt{m \log d}$ and
  $L_t = \lipobj / \smoothparam_t$. There is a universal
  constant $C$ such that
  \begin{align*}
    \E [f(x_T) + \regularizer(x_T)] - [f(x^*) + \regularizer(x^*)]
    & \le C \frac{\lipobj \radius \sqrt{d}}{T}
    + C \frac{\lipobj \radius \sqrt{\log d}}{\sqrt{Tm}} \\
    & = \order\left(\frac{\lipobj \lone{x^*} \sqrt{d \log d}}{T}
    + \frac{\lipobj \lone{x^*} \log d}{\sqrt{T m}}\right).
  \end{align*}
\end{corollary}
\noindent
The dimension dependence of $\sqrt{d \log d}$ on the leading
$1/T$ term in the corollary is weaker than the $d^{1/4}$ dependence in
the earlier corollaries, so for very large $m$ the corollary is not as strong
as one desires when taking advantage of non-Euclidean geometry.
Nonetheless, for large $T$, the $1 / \sqrt{Tm}$ terms dominate
the convergence rates, and Corollary~\ref{corollary:variance-rate-linf} can
be an improvement.

Our final corollary specializes the high probability convergence result in
Theorem~\ref{theorem:main-theorem-highprob-no-horizon} by showing that the
error is sub-Gaussian~\eqref{assumption:mgf-squared} under the assumptions
in the corollary. We state the corollary for problems with Euclidean geometry,
but it is clear that similar results hold for non-Euclidean geometry as above.
\begin{corollary}
  \label{corollary:highprob}
  Assume that $F(\cdot; \statsample)$ is $\lipobj$-Lipschitz with respect to
  the $\ell_2$-norm.  Let $\prox(x) = \half \ltwo{x}^2$ and assume that
  $\xdomain$ is compact with $\ltwo{x - x^*} \le \radius$ for $x, x^* \in
  \xdomain$. Using smoothing distribution $\smoothingdist$ uniform on $B_2(0,
  1)$, smoothing parameter $\smoothparam = \radius d^{1/4}$, damping parameter
  \mbox{$\extrastep_t = \lipobj \sqrt{t + 1}/ \radius \sqrt{m}$,} and
  instantaneous Lipschitz estimate $L_t = \lipobj \sqrt{d} / \smoothparam_t$,
  with probability at least $1 - \delta$,
  \begin{align*}
    \lefteqn{f(x_T) + \regularizer(x_T) - f(x^*) - \regularizer(x^*)} \\
    & = \order\Bigg(
    \frac{\lipobj \radius d^{1/4}}{T}
    + \frac{\lipobj \radius}{\sqrt{T m}}
    + \frac{\lipobj \radius \sqrt{\log \frac{1}{\delta}}}{\sqrt{Tm}}
    + \frac{\lipobj \radius
      \max\{\log\frac{1}{\delta}, \log T\}}{T\sqrt{m}}
    \Bigg).
  \end{align*}
\end{corollary}

\paragraph{Remarks:}
We make two remarks about the above corollaries. The first is that if one
abandons the requirement that the optimization procedure be an ``anytime''
algorithm---always able to return a result---it is possible to give similar
results by using a fixed setting of $\smoothparam_t \equiv \smoothparam$
throughout. In particular, using Theorem~\ref{theorem:main-theorem} in
Section~\ref{sec:main-proof} we can use $\smoothparam_t = u / T$
to get essentially the same results as
Corollaries~\ref{corollary:variance-rate-ltwo}--\ref{corollary:variance-rate-linf}. As
a side benefit, it is then easier to satisfy the Lipschitz condition that
$\E[\norm{\partial F(x; \statsample)}^2] \le \lipobj^2$ for $x \in \xdomain +
\supp \smoothingdist$.
Our second observation is that Theorem~\ref{theorem:main-theorem-no-horizon}
and the corollaries appear to require a very specific setting of the constant
$L_t$ to achieve fast rates. However, the algorithm is in fact robust to
mis-specification of $L_t$, since the instantaneous smoothness constant $L_t$
is dominated by the stochastic damping term $\extrastep_t$ in the algorithm.
Indeed, since $\extrastep_t$ grows proportionally to $\sqrt{t}$ for each
corollary, we always have $L_t = \lipgrad / \smoothparam_t = \lipgrad /
\theta_t u = \order(\extrastep_t / \sqrt{t} \theta_t)$; that is, $L_t$ is
order $\sqrt{t}$ smaller than $\extrastep_t / \theta_t$, so setting $L_t$
incorrectly up to order $\sqrt{t}$ has essentially negligible effect. \\

We can show the bounds in the theorems above are tight, that is, unimprovable
up to constant factors, by exploiting known lower
bounds~\cite{NemirovskiYu83,AgarwalBaRaWa12} for stochastic optimization
problems.  We re-state some of these results here.  For instance, let
$\xdomain = \{x \in \R^d \mid \ltwo{x} \le \radius_2\}$, and consider all
convex functions $f$ that are $\lipobjtwo$-Lipschitz with respect to the
$\ell_2$-norm. Assume that the stochastic oracle, when queried at a point $x$,
returns a vector $g$ whose expectation is in $\partial f(x)$ with
$\E[\ltwo{g}^2] \le \lipobjtwo^2$. Then for \emph{any} method that outputs a
point $x_T \in \xdomain$ after $T$ queries of the oracle, we have the lower
bound
\begin{equation*}
  \sup_f \Big\{\E[f(x_T)] - \min_{x \in \xdomain} f(x) \Big\}
  = \Omega\left(\frac{\lipobjtwo \radius_2}{\sqrt{T}}\right),
\end{equation*}
where the supremum is taken over convex $f$ that are
$\lipobjtwo$-Lipschitz with respect to the $\ell_2$-norm~\cite[Section
  3.1]{AgarwalBaRaWa12}.  Similar bounds hold for problems with
non-Euclidean geometry~\cite{AgarwalBaRaWa12};
in particular, consider convex $f$ that are
$\lipobjinf$-Lipschitz with respect to the $\ell_1$-norm, that is,
$|f(x) - f(y)| \le \lipobjinf \lone{x - y}$. Then setting $\xdomain =
\{x \in \R^d \mid \lone{x} \le \radius_1\}$, we have $B_\infty(0,
\radius_1/d) \subset B_1(0, \radius_1)$ and thus
\begin{equation*}
  \sup_f \Big\{\E[f(x_T)] - \min_{x \in \xdomain} f(x) \Big\}
  = \Omega\left(\frac{\lipobjinf \radius_1}{\sqrt{T}}\right).
\end{equation*}
In either geometry, no method can have optimization error smaller than
$\order(LR / \sqrt{T})$ after $T$ queries of the stochastic oracle.

Let us compare the above lower bounds to the convergence rates in
Corollaries~\ref{corollary:variance-rate-ltwo} through
\ref{corollary:variance-rate-linf}. Examining the bound in
Corollaries~\ref{corollary:variance-rate-ltwo}
and~\ref{corollary:variance-rate-normal}, we see that the dominant
terms are order $\lipobj \radius / \sqrt{Tm}$ so long as $m \le T /
\sqrt{d}$. Since our method issues $Tm$ queries to the oracle, this is
optimal. Similarly, the strategy of sampling uniformly from the
$\ell_\infty$-ball in Corollary~\ref{corollary:variance-rate-linf} is
optimal up to factors logarithmic in the dimension. In contrast to
other optimization procedures, however, our algorithm performs
an update to the parameter $x_t$ only after every $m$ queries to the
oracle; as we show in the next section, this is beneficial in several
applications.

\section{Applications and experimental results}
\label{sec:applications}

In this section, we describe some applications of our results, and
then give experimental results that illustrate our
theoretical predictions.

\subsection{Some applications}
\label{subsec:applications}

The first application of our results is to parallel
computation and distributed optimization. Imagine that instead of
querying the stochastic oracle serially, we can issue queries and
aggregate the resulting stochastic gradients in parallel. In
particular, assume that we have a procedure in which the
$m$ queries of the stochastic oracle occur concurrently. Then
Corollaries~\ref{corollary:variance-rate-ltwo}--\ref{corollary:highprob}
imply that in the same amount of time required to perform
$T$ queries and updates of the stochastic gradient oracle serially,
achieving an optimization error of $\order(1/ \sqrt{T})$, the
parallel implementation can process $Tm$ queries and consequently
has optimization error $\order(1 / \sqrt{Tm})$.

We now briefly describe two possibilities for a distributed implementation of
the above.  The simplest architecture is a master-worker architecture, in
which one master maintains the parameters $(x_t, y_t, z_t)$, and each of $m$
workers has access to an uncorrelated stochastic oracle for $\statprob$ and
the smoothing distribution $\smoothingdist$. The master broadcasts the point
$y_t$ to the workers, which sample $\statsample_i \sim \statprob$ and $Z_i
\sim \smoothingdist$, returning sample gradients to the master. In a
tree-structured network, broadcast and aggregation require at most
$\order(\log m)$ steps; the relative speedup over a serial implementation is
$\order(m / \log m)$.  In recent work, Dekel et al.~\cite{DekelGiShXi12} give
a series of reductions showing how to distribute variance-based stochastic
algorithms and achieve an asymptotically optimal convergence rate. The
algorithm given here, as specified by
equations~\eqref{eqn:average-smooth-gradient}
and~\eqref{eqn:y-update}--\eqref{eqn:x-update}, can be exploited within their
framework to achieve an $\order(m)$ improvement in convergence rate over a
serial implementation.  More precisely, whereas achieving optimization error
$\epsilon$ requires $\order(1/\epsilon^2)$ iterations for a centralized
algorithm, the distributed adaptation requires only $\order(1 / (m\epsilon^2))$
iterations. Such an improvement is possible as a consequence of the variance
reduction techniques we have described.

A second application of interest involves problems where the set
$\xdomain$ and/or the function $\regularizer$ are complicated, so that
calculating the proximal update~\eqref{eqn:z-update} becomes
expensive. The proximal update may be distilled to computing
\begin{equation}
  \label{eqn:projection}
  \min_{x \in \xdomain} ~\big \{ \<g, x\> + \prox(x) \big \} ~~~
  \mbox{or} ~~~ \min_{x \in \xdomain} ~ \big \{ \<g, x\> + \prox(x) +
  \regularizer(x) \big \}.
\end{equation}
In such cases, it may be beneficial to accumulate gradients by
querying the stochastic oracle several times in each iteration, using
the averaged subgradient in the update~\eqref{eqn:z-update}, and thus
solve only one proximal sub-problem for a collection of samples.

Let us consider some concrete examples.  In statistical applications
involving the estimation of covariance matrices, the domain $\xdomain$
is constrained in the positive semidefinite cone $\{X \in \S_n \mid X
\succeq 0\}$; other applications involve additional nuclear-norm
constraints of the form $\xdomain = \{X \in \R^{d_1 \times d_2} \mid
\sum_{j=1}^{\min\{d_1, d_2\}} \singval_j(X) \le C\}$.  Examples of
such problems include covariance matrix estimation, matrix completion,
and model identification in vector autoregressive processes (see the
paper~\cite{NegahbanWa10} and references therein for further
discussion).
Another example is the problem of metric
learning~\cite{XingNgJoRu03,ShalevSiNg04}, in which the learner is
given a set of $n$ points $\{a_1, \ldots, a_n\} \subset \R^d$ and a
matrix $B \in \R^{n \times n}$ indicating which points are close
together in an unknown metric. The goal is to estimate a positive
semidefinite matrix $X \succeq 0$ such that $\<(a_i - a_j), X(a_i -
a_j)\>$ is small when $a_i$ and $a_j$ belong to the same class or are
close, while $\<(a_i - a_j), X(a_i - a_j)\>$ is large when $a_i$ and
$a_j$ belong to different classes. It is desirable that the matrix $X$
have low rank, which allows the statistician to discover structure or
guarantee performance on unseen data.  As a concrete illustration,
suppose that we are given a matrix $B \in \{-1, 1\}^{n \times n}$,
where $b_{ij} = 1$ if $a_i$ and $a_j$ belong to the same class, and
$b_{ij} = -1$ otherwise.  In this case, one possible
optimization-based estimator involves solving the non-smooth program
\begin{equation}
  \label{eqn:metric-learning}
  \min_{X, x} ~ \frac{1}{\binom{n}{2}} \sum_{i < j} \hinge{1 +
    b_{ij}(\tr(X (a_i - a_j)(a_i - a_j)^\top) + x)} ~~ \mbox{s.t.} ~~
  X \succeq 0, ~ \tr(X) \le C.
\end{equation}
Now let us consider the cost of computing the projection
update~\eqref{eqn:projection} for the metric learning
problem~\eqref{eqn:metric-learning}. When $\prox(X) = \half
\frobenius{X}^2$, the update \eqref{eqn:projection} reduces for
appropriate choice of $V$ to
\begin{equation*}
  \min_X \half \frobenius{X - V}^2 ~~~ \mbox{subject to} ~~~
  X \succeq 0, ~ \tr(X) \le C.
\end{equation*}
(As a side-note, it is possible to generalize this update to Schatten
$p$-norms~\cite{DuchiShSiTe10}.)  The above problem is equivalent to
projecting the eigenvalues of $V$ to the simplex $\{x \in \R^d \mid x
\succeq 0, \<\onevec, x\> \le C\}$, which after an $\order(d^3)$
eigen-decomposition requires time $\order(d)$~\cite{Brucker84}. To see
the benefits of the randomized perturbation and averaging
technique~\eqref{eqn:average-smooth-gradient} over standard stochastic
gradient descent~\eqref{eqn:dual-averaging-update}, consider that the
cost of querying a stochastic oracle for the
objective~\eqref{eqn:metric-learning} for one sample pair $(i, j)$
requires time $\order(d^2)$. Thus, $m$ queries require $\order(md^2)$
computation, and each update requires $\order(d^3)$. So we see that
after $Tmd^2 + Td^3$ units of computation, our randomized perturbation
method has optimization error $\order(1 / \sqrt{Tm})$, while standard
stochastic gradient requires $Tmd^3$ units of computation.  In short,
for $m \approx d$ the randomized smoothing
technique~\eqref{eqn:average-smooth-gradient} uses a factor
$\order(d)$ less computation than standard stochastic gradient; we
give experiments corroborating this in
Section~\ref{sec:experiments-metric-learning}.

\subsection{Experimental results} 
\label{sec:experiments}

We now describe some experimental results that confirm the sharpness
of our theoretical predictions.  The first experiment explores the
benefit of using multiple samples $m$ when estimating the gradient
$\nabla f(y_t)$ as in the averaging
step~\eqref{eqn:average-smooth-gradient}. The second experiment
studies the actual amount of time required to solve a statistical
metric learning problem, as described in the
objective~\eqref{eqn:metric-learning} above. The third investigates
whether the smoothing technique is essential to algorithms solving
non-smooth stochastic problems---that is, whether the smoothing is
only a proof device or whether it is necessary to achieve good
performance.

\subsubsection{Iteration Complexity of Reduced Variance Estimators}
\label{sec:iteration-complexity}

\begin{figure}[ht!]
  \begin{center}
    \begin{tabular}{ccc}
          \includegraphics[width=.4 \textwidth]{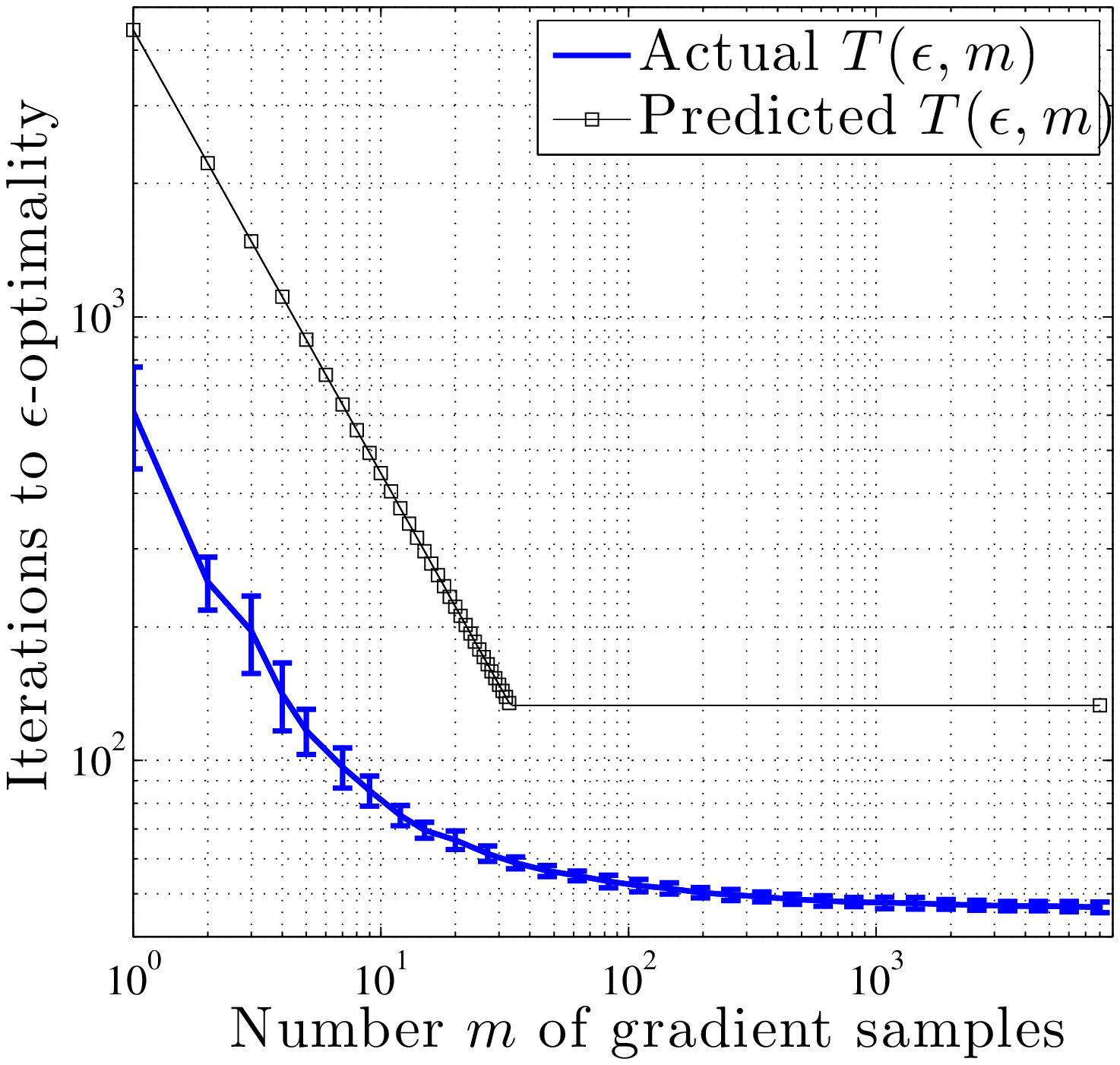} & &
          \includegraphics[width=.4 \textwidth]{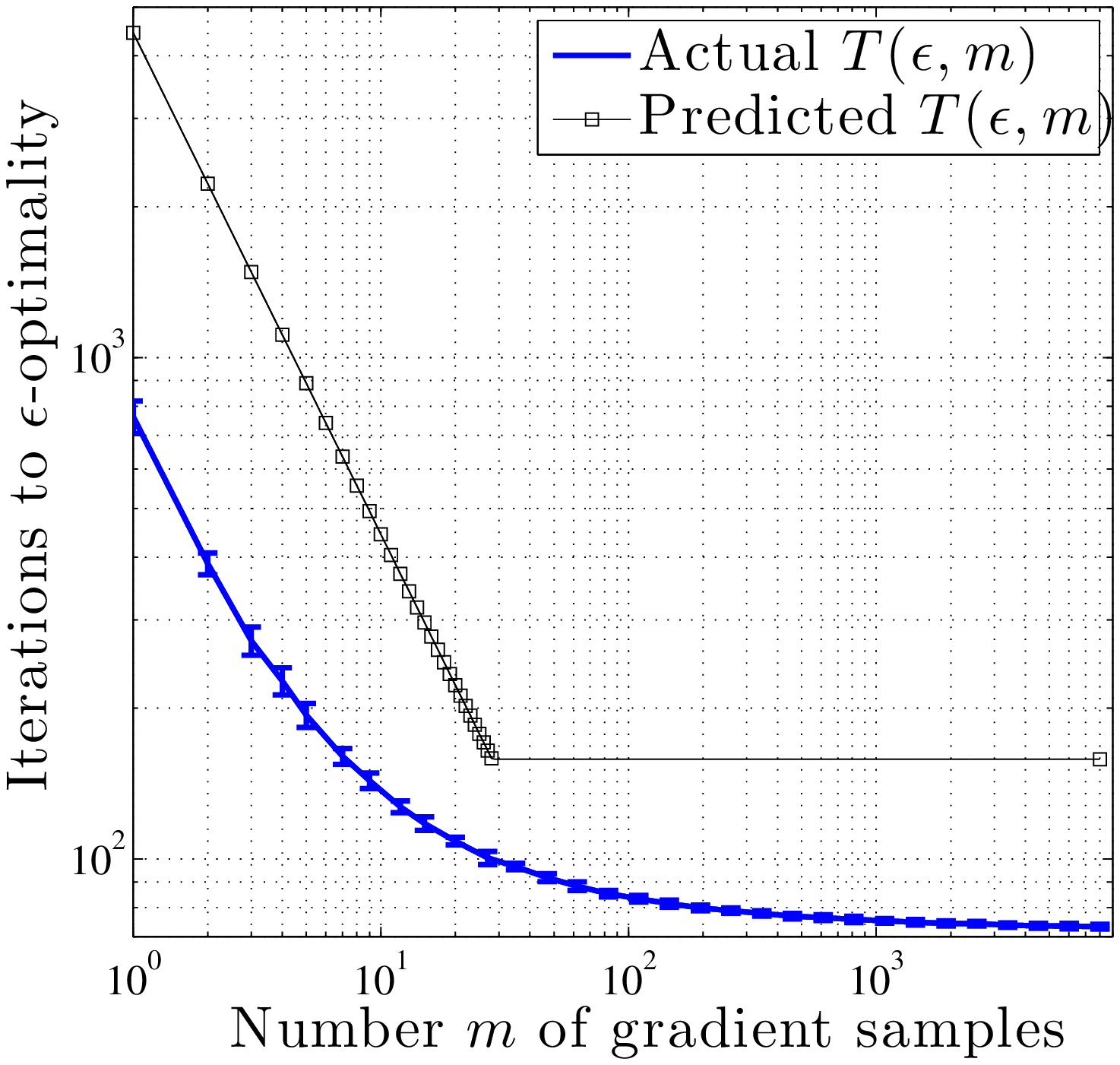} \\
          (a) & & (b) \\
          \includegraphics[width=.4 \textwidth]{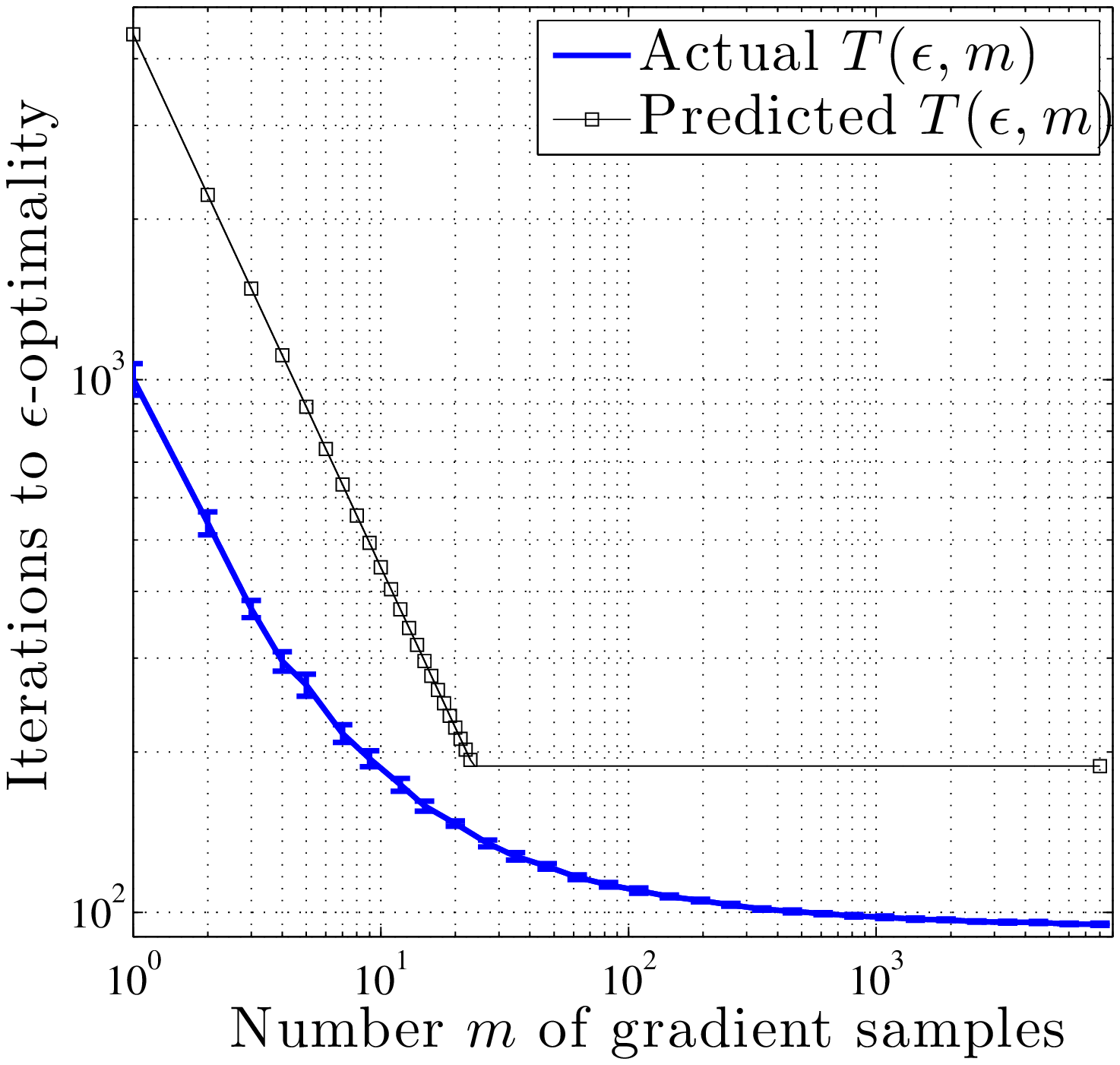} & &
          \includegraphics[width=.4 \textwidth]{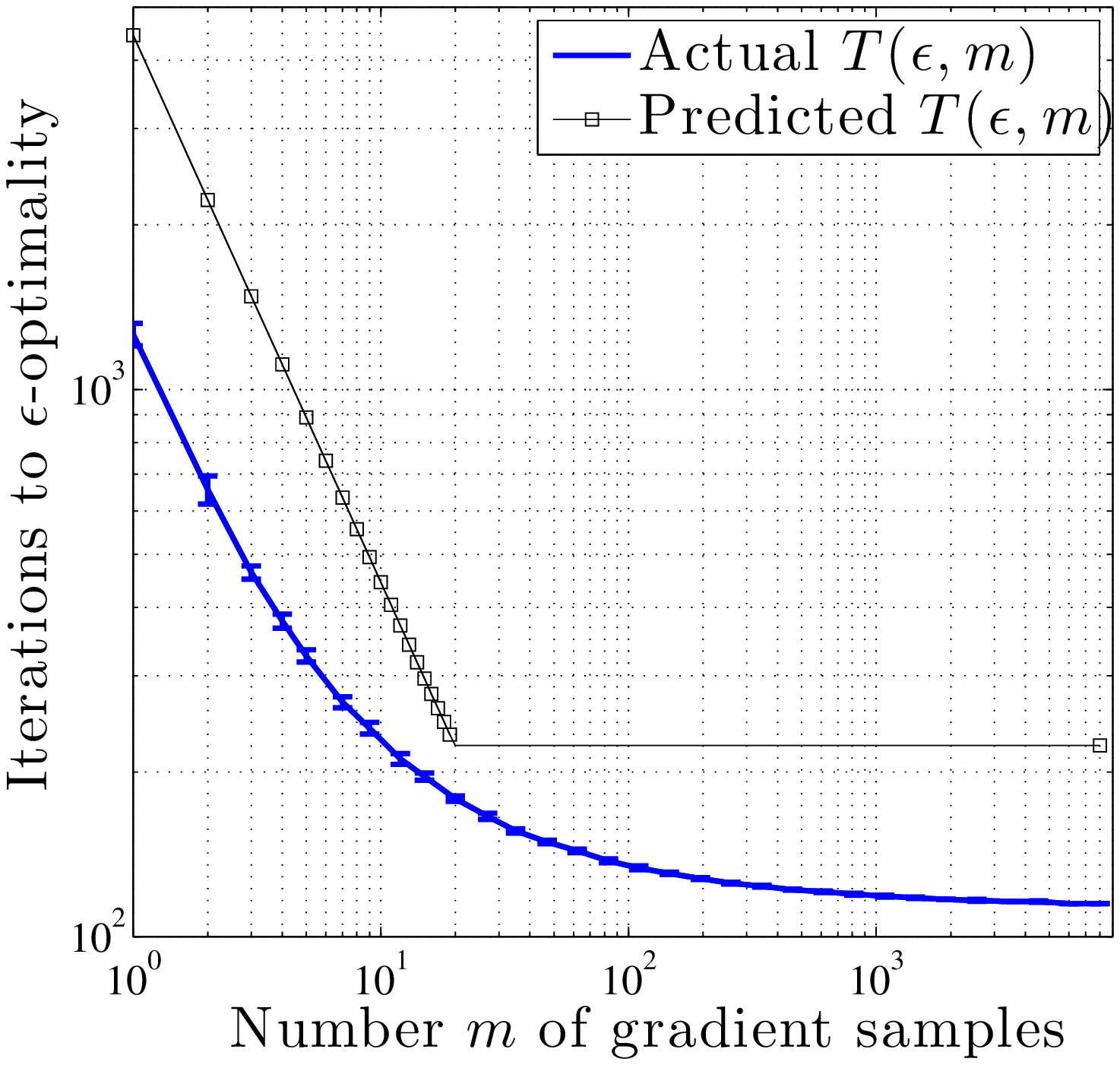} \\
          (c) & & (d)
        \end{tabular}
        \caption{\label{fig:time-to-epsilon} The number of iterations
          $T(\epsilon, m)$ to achieve an $\epsilon$-optimal solution for the
          problem \eqref{eqn:robust-linreg} as a function of the number of
          samples $m$ used in the gradient
          estimate~\eqref{eqn:average-smooth-gradient}. The
          prediction~\eqref{eqn:time-to-epsilon-ltwo-prediction}
          is the square black line in each plot, and each plot
          shows results for different dimensions $d$: (a) $d = 50$,
          (b) $d = 100$, (c) $d = 200$, (d) $d = 400$}
  \end{center}
\end{figure}


In this experiment, we consider the number of iterations of the
accelerated method \eqref{eqn:y-update}--\eqref{eqn:x-update} required
to achieve an $\epsilon$-optimal solution to the problem
\eqref{eqn:objective}. To understand how the iteration scales with the
number $m$ of gradient samples, it can be useful to consider our
results in terms of the number of iterations $T(\epsilon, m)$ required
to achieve optimization error $\epsilon$ for the optimization
procedure when using $m$ gradient samples in the averaging
step~\eqref{eqn:average-smooth-gradient}. Specifically, we define
\begin{equation*}
  T(\epsilon, m) = \inf\Big\{t \in \N \mid f(x_t) - \min_{x \in \xdomain}
  f(x^*) \le \epsilon \Big\}.
\end{equation*}
We focus on the algorithm analyzed in
Corollary~\ref{corollary:variance-rate-ltwo}, which uses uniform sampling of
the $\ell_2$-ball. The theorem implies there should be two regimes of
convergence: one when the number $m$ of samples is small, so that the
$\lipobj\radius / \sqrt{Tm}$ term is dominant, and the other when $m$ is
large, so the $\lipobj \radius d^{1/4} / T$ term is dominant. By inverting the first
term, we see that for small $m$, $T(\epsilon, m) = \order(\lipobj^2 \radius^2 / m
\epsilon^2)$, while the second gives $T(\epsilon, m) = \order(\lipobj \radius
d^{1/4} / \epsilon)$. In particular, our theory predicts
\begin{equation}
  T(\epsilon, m) = \order\left(\max\left\{
  \frac{\lipobj^2 \radius^2}{m \epsilon^2},
  \frac{\lipobj \radius d^{1/4}}{\epsilon} \right\}\right).
  \label{eqn:time-to-epsilon-ltwo-prediction}
\end{equation}

To assess the accuracy of the
prediction~\eqref{eqn:time-to-epsilon-ltwo-prediction}, we consider a
robust linear regression problem, commonly studied in system
identification and robust
statistics~\cite{PolyakTs80,Huber81}. Specifically, we have a matrix
$A \in \R^{n \times d}$ and vector $b \in \R^n$, and seek to minimize
\begin{equation}
  \label{eqn:robust-linreg}
  f(x) = \frac{1}{n} \lone{Ax - b} = \frac{1}{n} \sum_{i=1}^n |\<a_i,
  x\> - b_i|,
\end{equation}
where $a_i \in \R^d$ denotes a transposed row of $A$. It is clear that the
problem \eqref{eqn:robust-linreg} is non-smooth. The stochastic oracle in this
experiment, when queried at a point $x$, chooses an index $i \in [n]$
uniformly at random and returns $\sign(\<a_i, x\> - b_i) a_i$.

To perform our experiments, we generate $n = 1000$ points in dimensions
\mbox{$d \in \{50 \cdot 2^i\}_{i=0}^5$,} each with fixed norm
$\ltwo{a_i} = \lipobj$, and then assign values $b_i$ by computing $\<a_i,
w\>$ for a random vector $w$ (adding normally distributed noise with variance
$0.1$). We estimate the quantity $T(\epsilon, m)$ for solving the robust
regression problem~\eqref{eqn:robust-linreg} for several values of $m$ and
$d$.  Figure~\ref{fig:time-to-epsilon} shows results for dimensions \mbox{$d
  \in \{50, 100, 200, 400\}$}, averaged over 20 experiments for each choice of
dimension $d$.  (Other settings of $d$ exhibited similar behavior.)  Each
panel in the figure shows---on a log-log scale---the experimental average
$T(\epsilon, m)$ and the theoretical
prediction~\eqref{eqn:time-to-epsilon-ltwo-prediction}.  The decrease in
$T(\epsilon, m)$ is nearly linear for smaller numbers of samples $m$; for
larger $m$, the effect is quite diminished.  We present numerical results in
Table~\ref{table:timing} that allow us to roughly estimate the number $m$ at
which increasing the batch size in the gradient
estimate~\eqref{eqn:average-smooth-gradient} gives no further improvement. For
each dimension $d$, Table~\ref{table:timing} indeed shows that from $m = 1$ to
$5$, the iteration count $T(\epsilon, m)$ decreases linearly, halving again
when we reach $m \approx 20$, but between $m = 100$ and $1000$ there is at
most an 11\% difference in $T(\epsilon, m)$, while between $m = 1000$ and $m =
10000$ the decrease amounts to at most 3\%.  The good qualitative match
between the iteration complexity predicted by our theory and the actual
performance of the methods is clear.

\begin{table}
  \begin{center}
  \begin{tabular}{|cc|c|c|c|c|c|c|c|c|}
    \hline
    & $m$ & 1 & 2 & 3 & 5 & 20 & 100 & 1000 & 10000 \\
    \hline
    \multirow{2}{*}{$d = 50$}
    & \small{\textsc{Mean}} & 612.2 &252.7 &195.9 &116.7 &66.1 &52.2 &47.7 &46.6 \\
    & \small{\textsc{Std}} & 158.29 &34.67 &38.87 &13.63 &3.18 &1.66 &1.42 &1.28 \\
    \hline
    \multirow{2}{*}{$d = 100$}
    & \small{\textsc{Mean}} & 762.5 &388.3 &272.4 &193.6 &108.6 &83.3 &75.3 &73.3 \\
    & \small{\textsc{Std}} & 56.70 &19.50 &17.59 &10.65 &1.91 &1.27 &0.78 &0.78 \\
    \hline
    \multirow{2}{*}{$d = 200$}
    & \small{\textsc{Mean}} & 1002.7 &537.8 &371.1 &267.2 &146.8 &109.8 &97.9 &95.0 \\
    & \small{\textsc{Std}} & 68.64 &26.94 &13.75 &12.70 &1.66 &1.25 &0.54 &0.45 \\
    \hline
    \multirow{2}{*}{$d = 400$}
    & \small{\textsc{Mean}} & 1261.9 &656.2 &463.2 &326.1 &178.8 &133.6 &118.6 &115.0 \\
    & \small{\textsc{Std}} & 60.17 &38.59 &12.97 &8.36 &2.04 &1.02 &0.49 &0.00 \\
    \hline
    \multirow{2}{*}{$d = 800$}
    & \small{\textsc{Mean}} & 1477.1 &783.9 &557.9 &388.3 &215.3 &160.6 &142.0 &137.4 \\
    & \small{\textsc{Std}} & 44.29 &24.87 &12.30 &9.49 &2.90 &0.66 &0.00 &0.49 \\
    \hline
    \multirow{2}{*}{$d \hspace{-.029cm} = \hspace{-.029cm}
      1\hspace{-.015cm}600$}
    & \small{\textsc{Mean}} & 1609.5 &862.5 &632.0 &448.9 &251.5 &191.1 &169.4 &164.0 \\
    & \small{\textsc{Std}} & 42.83 &30.55 &12.73 &8.17 &2.73 &0.30 &0.49 &0.00 \\
    \hline
  \end{tabular}
  \end{center}
  \caption{\label{table:timing} The number of iterations $T(\epsilon,
    m)$ to achieve $\epsilon$-accuracy for the regression
    problem~\eqref{eqn:robust-linreg} as a function of number of
    gradient samples $m$ used in the gradient
    estimate~\eqref{eqn:average-smooth-gradient} and the dimension
    $d$.  Each box in the table shows the mean and standard deviation
    of $T(\epsilon, m)$ measured over $20$ trials.}
\end{table}

\subsubsection{Metric Learning}
\label{sec:experiments-metric-learning}

\begin{figure}[t]
  \begin{center}
    \includegraphics[width=.5 \textwidth]{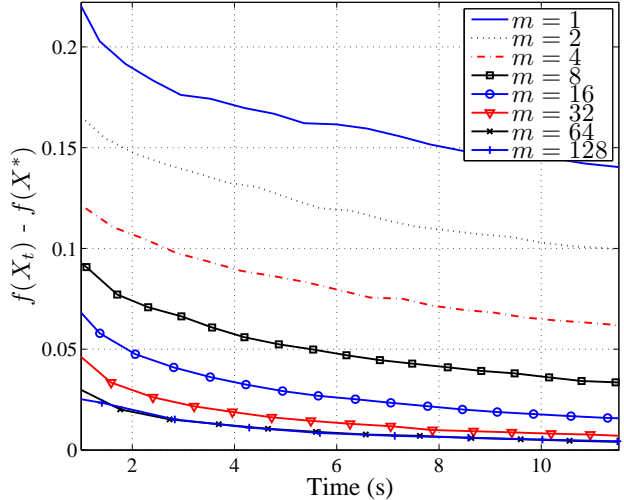}
    \caption{\label{fig:metric-learning} Optimization error $f(X_t) -
      \inf_{X^* \in \xdomain}f(X^*)$ in the metric learning problem of
      Sec.~\ref{sec:experiments-metric-learning} as a function of time in
      seconds. Each line indicates optimization error over time for a
      particular number of samples $m$ in the gradient
      estimate~\eqref{eqn:average-smooth-gradient}; we set $m = 2^i$ for $i =
      \{1, \ldots, 7\}$.}
  \end{center}
\end{figure}

Our second set of experiments apply to instances of metric learning.
The data we receive consists of a set of vectors $a_i \in \R^d$ and
measures $b_{ij} \ge 0$ of the similarity between the vectors $a_i$
and $a_j$ (here $b_{ij} = 0$ means that $a_i$ and $a_j$ are the
same). The statistical goal is to learn a matrix $X$---inducing a
pseudo-norm via $\norm{a}_X^2 \defeq \<a, X a\>$---such that $\<(a_i -
a_j), X(a_i - a_j)\> \approx b_{ij}$. Consequently, we solve the
regression-like problem
\begin{equation*}
  f(X) = \frac{1}{\binom{n}{2}} \sum_{i < j}
  \left|\tr\left(X(a_i - a_j)(a_i - a_j)^\top\right) - b_{ij}\right|
  ~~~ \mbox{subject to} ~~~
  \tr(X) \le C, ~ X \succeq 0.
\end{equation*}
The stochastic oracle for this problem is simple: given a query matrix $X$,
the oracle chooses a pair $(i, j)$ uniformly at random, then returns the
subgradient
\begin{equation*}
  \sign\left[\<(a_i - a_j), X(a_i - a_j)\> - b_{ij}\right]
  (a_i - a_j)(a_i - a_j)^\top.
\end{equation*}
We solve ten random problems with dimension $d = 100$ and $n = 2000$, giving
an objective with $4 \cdot 10^6$ terms and $5050$ parameters.  We plot
experimental results in Fig.~\ref{fig:metric-learning} showing the optimality
gap $f(X_t) - \inf_{X^* \in \xdomain} f(X^*)$ as a function of computation
time. We plot several lines, each of which captures the performance of the
algorithm using a different number $m$ of samples in the smoothing step
\eqref{eqn:average-smooth-gradient}. It is clear that performing stochastic
optimization is more viable for this problem than a non-stochastic method, as
even computing the objective requires $\order(n^2 d^2)$ operations. As
predicted by our theory and discussion in Sec.~\ref{sec:applications}, it is
clear that receiving more samples $m$ gives improvements in convergence rate
as a function of time. Our theory also predicts that for $m \ge d$, there
should be no improvement in actual time taken to minimize the objective; the
plot in Fig.~\ref{fig:metric-learning} suggests that this too is correct, as
the plots for $m = 64$ and $m = 128$ are essentially indistinguishable.

\subsubsection{Necessity of randomized smoothing}

\begin{figure}
  \begin{center}
    \includegraphics[width= 0.5\textwidth]{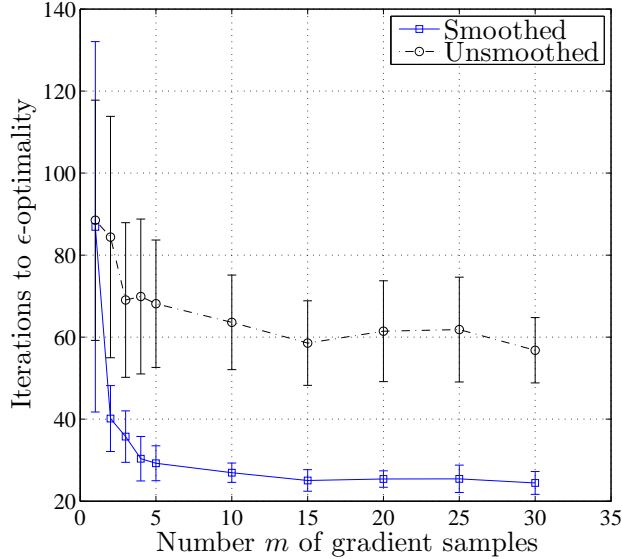}
    \caption{\label{fig:necessity-smoothing} The number of iterations
      $T(\epsilon, m)$ to achieve an $\epsilon$-optimal solution to
      \eqref{eqn:l1-centroid} for simple mirror descent and the smoothed
      gradient method.}
  \end{center}
\end{figure}

A reasonable question is whether the extra sophistication of the
random smoothing~\eqref{eqn:average-smooth-gradient} is necessary. Can
receiving more samples $m$ from the stochastic oracle---all evaluated
at the same point---give the same benefit to the simple dual averaging
method~\eqref{eqn:dual-averaging-update}?  We do not know the full
answer to this question, though we give an experiment here that
suggests that the answer is negative, in that smoothing does give
demonstrable improvement.

For this experiment, we use the objective
\begin{equation}
  \label{eqn:l1-centroid}
  f(x) = \frac{1}{n} \sum_{i=1}^n \lone{x - a_i},
\end{equation}
where the $a_i \in \{-1, +1\}^d$, and each component $j$ of the vector $a_i$
is sampled independently from $\{-1, 1\}$ and equal to $1$ with probability $1
/ \sqrt{j}$. Even as $n \uparrow \infty$, the function $f$ remains non-smooth,
since the $a_i$ belong to a discrete set and each value of $a_i$ occurs with
positive probability. As in Sec.~\ref{sec:iteration-complexity}, we compute
$T(\epsilon, m)$ to be the number of iterations required to achieve an
$\epsilon$-optimal solution to the objective~\eqref{eqn:l1-centroid}. We
compare two algorithms that use $m$ queries of the stochastic gradient oracle,
which when queried at a point $x$ chooses an index $i \in [n]$ uniformly at
random and returns $\sign(x - a_i) \in \partial \lone{x - a_i}$. The first
algorithm is the dual averaging algorithm~\eqref{eqn:dual-averaging-update},
where $g_t$ is the average of $m$ queries to the stochastic oracle at the
current iterate $x_t$. The second is the accelerated method
\eqref{eqn:y-update}--\eqref{eqn:x-update} with the randomized
averaging~\eqref{eqn:average-smooth-gradient}.  We plot the results in
Fig.~\ref{fig:necessity-smoothing}. We plot the best stepsize sequence
$\stepsize_t$ for the update~\eqref{eqn:dual-averaging-update} of several we
tested to make comparison as favorable as possible for simple mirror
descent. It is clear that while there is moderate improvement for the
non-smooth method when the number of samples $m$ grows, and both methods are
(unsurprisingly) essentially indistinguishable for $m = 1$, the smoothed
sampling strategy has much better iteration complexity as $m$ grows.


\section{Proofs} 
\label{sec:main-result-proofs}

In this section, we provide the proofs of
Theorems~\ref{theorem:main-theorem-no-horizon}
and~\ref{theorem:main-theorem-highprob-no-horizon}, as well as
Corollaries~\ref{corollary:variance-rate-ltwo}
through~\ref{corollary:highprob}. We begin with the proofs of the
corollaries, after which we give the full proofs of the theorems. 
In both cases, we defer some of the more technical lemmas to appendices.

The general technique for the proof of each corollary is as
follows. First, we recognize that the randomly smoothed function
$\smoothedfunc(x) = \E f(x + Z)$ for $Z \sim \smoothingdist$ has
Lipschitz continuous gradients and is uniformly close to the original
non-smooth function $f$. This allows us to apply
Theorems~\ref{theorem:main-theorem}
or~\ref{theorem:main-theorem-no-horizon}. The second step is to
realize that with the sampling
procedure~\eqref{eqn:average-smooth-gradient}, the variance $\E
\dnorm{\error_t}^2$ decreases at a rate of approximately $1/m$, the
number of gradient samples. Choosing the stepsizes appropriately in
the theorems then completes the proofs.  Proofs of these corollaries
require relatively tight control of the smoothness properties of the
smoothing convolution~\eqref{eqn:convolve}, so we refer
frequently to several lemmas stated in
Appendix~\ref{appendix:smoothing}.

\subsection{Proof of Corollaries~\ref{corollary:variance-rate-ltwo}
 and~\ref{corollary:variance-rate-normal}}

We begin by proving Corollary~\ref{corollary:variance-rate-ltwo}.  Recall the
averaged quantity $g_t = \frac{1}{m} \sum_{i=1}^m g_{i,t}$, and that $g_{i,t}
\in \partial F(y_t + \smoothparam_t Z_i; \statsample_i)$, where the random
variables $Z_i$ are distributed uniformly on the ball $B_2(0, 1)$.
From Lemma~\ref{lemma:ltwo} in Appendix~\ref{appendix:smoothing}, the
variance of $g_t$ as an estimate of $\nabla \smoothedfunc[t](y_t)$ satisfies
\begin{align}
  \label{EqnBanffMBound}
  \stddev^2 \, \defn \, \E[\ltwo{e_t}^2] \; = \; \E[\ltwo{g_t - \nabla
    \smoothedfunc[t](y_t)}^2] & \le \; \frac{\lipobj^2}{m}.
\end{align}
Further, for $Z$ distributed uniformly on $B_2(0, 1)$,
we have the bound
\begin{equation*}
  f(x) \le \E[f(x + \smoothparam Z)] \le f(x) + \lipobj \smoothparam,
\end{equation*}
and moreover, the function $x \mapsto \E_\smoothingdist[f(x + \smoothparam Z)]$ has $\lipobj \sqrt{d}/
\smoothparam$-Lipschitz continuous gradient.  Thus, applying
Lemma~\ref{lemma:ltwo} and
Theorem~\ref{theorem:main-theorem-no-horizon} with the setting $L_t =
\lipobj \sqrt{d} / u \theta_t$, we obtain
\begin{equation*}
\E[f(x_T) + \regularizer(x_T)] - [f(x^*) + \regularizer(x^*)] \le
\frac{6\lipobj \radius^2 \sqrt{d}}{T u} + \frac{2 \extrastep_T
    \radius^2}{T} + \frac{1}{T} \sum_{t=0}^{T-1}
    \frac{1}{\extrastep_t} \cdot \frac{ \lipobj^2}{m} + \frac{4\lipobj
    u}{T},
  \end{equation*}
where we have used the bound~\eqref{EqnBanffMBound}.

Recall that $\extrastep_t = \lipobj \sqrt{t + 1} / \radius
\sqrt{m}$ by construction.  Coupled with the inequality
\begin{equation}
  \label{eqn:sqrt-sum-bound}
  \sum_{t=1}^{T} \frac{1}{\sqrt{t}} \le 1 + \int_1^T
  \frac{1}{\sqrt{t}} dt = 1 + 2(\sqrt{T} - 1) \le 2 \sqrt{T},
\end{equation}
we use that $2 \sqrt{T + 1} / T + 2 / \sqrt{T} \le 5 / \sqrt{T}$ to obtain
\begin{equation*} 
  \E[f(x_T) + \regularizer(x_T)] - [f(x^*) + \regularizer(x^*)] \le
  \frac{6 \lipobj \radius^2 \sqrt{d}}{T \smoothparam} + \frac{5
    \lipobj \radius}{\sqrt{Tm}} + \frac{4 \lipobj \smoothparam}{T}.
\end{equation*}
Substituting the specified setting of $\smoothparam = \radius d^{1/4}$
completes the proof. \\

The proof of Corollary~\ref{corollary:variance-rate-normal} is essentially
identical, differing only in the setting of $\smoothparam = \radius d^{-1/4}$
and the application of Lemma~\ref{lemma:normal} instead of
Lemma~\ref{lemma:ltwo} in Appendix~\ref{appendix:smoothing}.

\subsection{Proof of Corollary~\ref{corollary:variance-rate-linf}}

Under the stated conditions of the corollary, Lemma~\ref{lemma:linf}
implies that when $\smoothingdist$ is uniform on $B_\infty(0,
\smoothparam)$, then the function $\smoothedfunc(x) \defn
\E_\smoothingdist[f(x + Z)]$ has $\lipobj / \smoothparam$-Lipschitz
continuous gradient with respect to the $\ell_1$-norm, and moreover it
satisfies the upper bound $\smoothedfunc(x) \le f(x) +
\frac{\lipobj d \smoothparam}{2}$.
Fix $x \in \xdomain$ and let $g_i \in \partial F(x + Z_i;
\statsample_i)$, with $g = \frac{1}{m}\sum_{i=1}^m g_i$. We claim that
for any $\smoothparam$ the
error satisfies
\begin{align}
  \label{EqnAirplane}
  \E\big[\!\linf{g - \nabla \smoothedfunc(x)}^2\!\big]
  & \leq C \; \frac{\lipobj^2 \log d}{m}
\end{align}
for some universal constant $C$.  Indeed, Lemma~\ref{lemma:linf} shows that
$\E[g] = \nabla \smoothedfunc(x)$; moreover, component $j$ of the
random vector $g_i$ is an unbiased estimator of the $j$th component of
$\nabla \smoothedfunc(x)$.  Since $\linf{g_i} \le \lipobj$ and
$\linf{\nabla \smoothedfunc(x)} \le \lipobj$, the vector $g_i - \nabla
\smoothedfunc(x)$ is a $d$-dimensional random vector whose
components are sub-Gaussian with sub-Gaussian parameter
$4\lipobj^2$. Conditional on $x$, the $g_i$ are independent, so
$g - \nabla \smoothedfunc(x)$ has sub-Gaussian components with parameter
at most $4 \lipobj^2/m$. Applying Lemma~\ref{lemma:sub-gaussian} (see
Appendix~\ref{AppTail}) with $X = g - \nabla \smoothedfunc(x)$ and
$\sigma^2 = 4 \lipobj^2 / m$ yields the claim~\eqref{EqnAirplane}.

Now, as in the proof of Corollary~\ref{corollary:variance-rate-ltwo},
we can apply Theorem~\ref{theorem:main-theorem-no-horizon}.  Recall
that $\frac{1}{2(p - 1)} \norm{x}_p^2$ is strongly convex over $\R^d$
with respect to the $\ell_p$-norm for any $p \in
\openleft{1}{2}$~\cite{NemirovskiYu83}. Thus, with the choice
$\prox(x) = \frac{1}{2(p - 1)} \norm{x}_p^2$ for $p = 1 + 1/ \log d$,
it is clear that the squared radius $\radius^2$ of the set $\xdomain$
is order $\norm{x^*}_p^2 \log d \le \lone{x^*}^2 \log d$.
All that remains is to relate the Lipschitz constant
$\lipobj$ with respect to the $\ell_1$ norm to that for the $\ell_p$
norm. Let $q$ be conjugate to $p$, that is, $1/q + 1/p = 1$. Under the
assumptions of the theorem, we have $q = 1 + \log d$. For any $g \in
\R^d$, we have
$\norm{g}_q \le d^{1/q} \linf{g}$. Of course, $d^{1 / (\log d + 1)}
\le d^{1 / (\log d)} = \exp(1)$, so $\norm{g}_q \le e \linf{g}$.

Having shown that the Lipschitz constant $L$ for the $\ell_p$ norm
satisfies $L \le \lipobj e$, where $\lipobj$ is the Lipschitz
constant with respect to the $\ell_1$ norm, we apply
Theorem~\ref{theorem:main-theorem-no-horizon} and the variance
bound~\eqref{EqnAirplane} to obtain the result. Specifically,
Theorem~\ref{theorem:main-theorem-no-horizon} implies
\begin{equation*}
  \E[f(x_T) + \regularizer(x_T)] - [f(x^*) + \regularizer(x^*)]
  \le \frac{6 \lipobj \radius^2}{T \smoothparam}
  + \frac{2 \extrastep_T \radius^2}{T}
  + \frac{C}{T} \sum_{t=0}^{T-1} \frac{1}{\extrastep_t}
  \cdot \frac{\lipobj^2 \log d}{m} + \frac{4 \lipobj d \smoothparam}{2T}.
\end{equation*}
Plugging in $\smoothparam$, $\extrastep_t$, and
$\radius \le \lone{x^*} \sqrt{\log d}$ and using
bound~\eqref{eqn:sqrt-sum-bound} completes the proof.

\subsection{Proof of Corollary~\ref{corollary:highprob}}

The proof of this corollary requires an auxiliary result showing that
Assumption~\ref{assumption:subexp-errors} holds under the stated conditions.
The following result does not appear to be well-known, so we provide a proof
in Appendix~\ref{appendix:proof-of-lipschitz-concentration}.  In stating it,
we recall the definition of the sigma field $\mathcal{F}_t$ from
Assumption~\ref{assumption:subexp-errors}.
\begin{lemma}
  \label{lemma:lipschitz-sample-bound}
  In the notation of
  Theorem~\ref{theorem:main-theorem-highprob-no-horizon}, suppose
  that $F(\cdot; \statsample)$ is $\lipobj$-Lipschitz continuous with
  respect to the norm $\norm{\cdot}$ over $\xdomain + \smoothparam_0 \supp
  \smoothingdist$ for $\statprob$-a.e.\ $\statsample$. Then
  \begin{equation*}
    \E\bigg[\exp\bigg(\frac{\dnorm{\error_t}^2}{\sigma^2}\bigg) \mid
      \mc{F}_{t-1} \bigg] \le \exp(1),
    ~~~ \mbox{where} ~~~
    \stddev^2 \defn
    2 \max\bigg\{\E[\dnorm{\error_t}^2 \mid \mc{F}_{t-1}],
    \: \frac{16 \lipobj^2}{m} \bigg\}.
  \end{equation*}
\end{lemma}
\noindent

Using this lemma, we now prove Corollary~\ref{corollary:highprob}.  When
$\smoothingdist$ is the uniform distribution on $B_2(0, \smoothparam)$,
Lemma~\ref{lemma:ltwo} from Appendix~\ref{appendix:smoothing} implies that
$\nabla \smoothedfunc$ is Lipschitz with constant $\lipgrad = \lipobj
\sqrt{d} / u$.
Lemma~\ref{lemma:lipschitz-sample-bound} ensures that the error $e_t$
satisfies Assumption~\ref{assumption:subexp-errors}.  Noting the inequality
\begin{equation*}
  \max\left\{\log (1/\delta), \sqrt{(1 + \log T) \log (1 / \delta)}\right\} \le
  \max\{\log(1/\delta), 1 + \log T\}
\end{equation*}
and combining the bound in
Theorem~\ref{theorem:main-theorem-highprob-no-horizon} with
Lemma~\ref{lemma:lipschitz-sample-bound}, we see that
with probability at least $1 - 2\delta$
\begin{align*}
  \lefteqn{f(x_T) + \regularizer(x_T) - f(x^*) - \regularizer(x^*)} \\
  & \le \frac{6 \lipobj \radius^2 \sqrt{d}}{T \smoothparam}
  + \frac{4 \lipobj \smoothparam}{T}
  + \frac{4 \radius^2 \extrastep}{\sqrt{T + 1}}
  + \frac{2 \lipobj^2}{m \sqrt{T} \extrastep}
  + C \frac{\lipobj^2 \max\left\{\log\frac{1}{\delta}, \log T\right\}}{
    (T + 1) m \extrastep}
  + \frac{\lipobj \radius \sqrt{\log \frac{1}{\delta}}}{\sqrt{Tm}}
\end{align*}
for a universal constant $C$.
Setting $\extrastep = \lipobj / \radius \sqrt{m}$ and
$\smoothparam = \radius d^{1/4}$ gives the result.

\subsection{Proof of Theorem~\ref{theorem:main-theorem-no-horizon}}
\label{sec:main-proof}

This proof is more involved than that of the above corollaries. In particular,
we build on techniques used in the work of Tseng~\cite{Tseng08},
Lan~\cite{Lan10}, and Xiao~\cite{Xiao10}. The changing smoothness of the
stochastic objective---which comes from changing the shape parameter of the
sampling distribution $Z$ in the averaging
step~\eqref{eqn:average-smooth-gradient}---adds some challenge.  Essentially,
the idea of the proof is to let $\smoothingdist_t$ be the density of
$\smoothparam_t Z$ and define $\smoothedfunc[t](x) \defeq
\E_\smoothingdist[f(x + \smoothparam_t Z)]$, where $\smoothparam_t$ is the
non-increasing sequence of shape parameters in the averaging
scheme~\eqref{eqn:average-smooth-gradient}. We show via Jensen's inequality
that $f(x) \le \smoothedfunc[t](x) \le \smoothedfunc[t-1](x)$ for all $t$,
which is intuitive because the variance of the sampling scheme is
decreasing. Then we apply a suitable modification of the accelerated gradient
method~\cite{Tseng08} to the sequence of functions $\smoothedfunc[t]$
decreasing to $f$, and by allowing $\smoothparam_t$ to decrease appropriately
we achieve our result.  At the end of this section, we state a third result
(Theorem~\ref{theorem:main-theorem}), which gives an alternative setting for
$\smoothparam$ given a priori knowledge of the number of iterations.

We begin by stating two technical lemmas:
\begin{lemma}
  \label{lemma:variance-sequence-rda}
  Let $\smoothedfunc[t]$ be a sequence of functions such that $\smoothedfunc[t]$ has $L_t$-Lipschitz
  continuous gradients with respect to the norm $\norm{\cdot}$ and assume that
  $\smoothedfunc[t](x) \le \smoothedfunc[t-1](x)$ for any $x \in \xdomain$.  Let the sequence $\{x_t,
  y_t, z_t\}$ be generated according to the
  updates~\eqref{eqn:y-update}--\eqref{eqn:x-update}, and define the error
  term $\error_t = \nabla \smoothedfunc[t](y_t) - g_t$. Then for any $x^* \in \xdomain$,
  \begin{align*}
    \frac{1}{\theta_t^2} [\smoothedfunc[t](x_{t+1}) + \regularizer(x_{t+1})]
    & \le \sum_{\tau=0}^t \frac{1}{\theta_\tau}[
      \smoothedfunc[\tau](x^*) + \regularizer(x^*)]
    + \left(L_{t+1} + \frac{\extrastep_{t+1}}{\theta_{t+1}}\right) \prox(x^*)
    \\
    & \qquad \quad~ +
    \sum_{\tau=0}^t \frac{1}{2 \theta_\tau \extrastep_\tau} \dnorm{\error_t}^2
    + \sum_{\tau=0}^t \frac{1}{\theta_\tau} \<\error_\tau,
    z_\tau - x^*\>.
  \end{align*}
\end{lemma}
\noindent See Appendix~\ref{AppLemVarRDA} for the proof of this claim.

\begin{lemma}
  \label{lemma:sequence-props}
  Let the sequence $\theta_t$ satisfy $\frac{1 - \theta_t}{\theta_t^2} =
  \frac{1}{\theta_{t-1}^2}$ and $\theta_0 = 1$. Then $\theta_t \le \frac{2}{t
    + 2}$, and $\sum_{\tau=0}^t \frac{1}{\theta_\tau} =
  \frac{1}{\theta_t^2}$.
\end{lemma}
\noindent The second statement was proved by Tseng~\cite{Tseng08};
the first follows by a straightforward induction. \\


We now proceed with the proof.  Recalling $\smoothedfunc[t](x) = \E[f(x +
  \smoothparam_t Z)]$, let us verify that $\smoothedfunc[t](x) \leq
\smoothedfunc[t-1](x)$ for any $x$ and $t$ so we can apply
Lemma~\ref{lemma:variance-sequence-rda}.  Since $\smoothparam_t \leq
\smoothparam_{t-1}$, we may define a random variable $U \in \{0, 1\}$ such
that $\P(U = 1) = \frac{\smoothparam_t}{\smoothparam_{t-1}} \in [0,1]$. Then
\begin{align*}
   \smoothedfunc[t](x) \, = \, \E[f(x + \smoothparam_t Z)]
   & = \E\big[f \big(x + \smoothparam_{t-1} Z \E [U]
   \big)\big] \\
   & \le \P[U = 1] \; \E[f(x + \smoothparam_{t-1} Z)] + \P[U = 0] \;
   f(x),
\end{align*}
where the inequality follows from Jensen's inequality.  By a second
application of Jensen's inequality, we have $f(x) = f(x + \smoothparam_{t-1}
\E[Z]) \leq \E[f(x + \smoothparam_{t-1} Z)] = \smoothedfunc[t-1](x)$.  Combined
with the previous inequality, we conclude that $\smoothedfunc[t](x) \leq
\E[f(x + \smoothparam_{t-1} Z)] \; = \; \smoothedfunc[t-1](x)$ as claimed.
Consequently, we have verified that the function $\smoothedfunc[t]$ satisfies
the assumptions of Lemma~\ref{lemma:variance-sequence-rda} where $\nabla
\smoothedfunc[t]$ has Lipschitz parameter $L_t = \lipgrad / \smoothparam_t$
and error term $\error_t = \nabla \smoothedfunc[t](y_t) - g_t$.  We apply the
lemma momentarily.

Using Assumption~\ref{assumption:smoothness} that $f(x) \ge \E[f(x +
  \smoothparam_t Z)] - \lipobj \smoothparam_t = \smoothedfunc[t](x) - \lipobj
\smoothparam_t$ for all $x \in \xdomain$, Lemma~\ref{lemma:sequence-props}
implies
\begin{align}
  \lefteqn{\frac{1}{\theta_{T-1}^2} [f(x_T) + \regularizer(x_T)] -
    \frac{1}{\theta_{T-1}^2}[f(x^*) + \regularizer(x^*)]} \nonumber \\
  & \quad = \frac{1}{\theta_{T-1}^2} [f(x_T) + \regularizer(x_T)] -
  \sum_{t=0}^{T-1} \frac{1}{\theta_t}[f(x^*) + \regularizer(x^*)]
  \nonumber \\
  & \quad \le \frac{1}{\theta_{T-1}^2}[\smoothedfunc[t-1](x_T) +
    \regularizer(x_T)]
  - \sum_{t=0}^{T - 1} \frac{1}{\theta_t}[\smoothedfunc[t](x^*)
    + \regularizer(x^*)]
  + \sum_{t=0}^{T-1} \frac{\lipobj \smoothparam_t}{\theta_t}
  \nonumber,
\end{align}
which by the definition of $\smoothparam_t$ is in turn bounded by
\begin{equation}
  \frac{1}{\theta_{T-1}^2}[\smoothedfunc[t-1](x_T) +
  \regularizer(x_T)] - \sum_{t=0}^{T-1} \frac{1}{\theta_t}[\smoothedfunc[t](x^*) +
  \regularizer(x^*)] + T \lipobj u.
    \label{eqn:apply-variance-sequence}
\end{equation}
Now we
simply apply Lemma~\ref{lemma:variance-sequence-rda} to the
bound~\eqref{eqn:apply-variance-sequence}, which gives us
\begin{align}
\lefteqn{\frac{1}{\theta_{T-1}^2} \left[f(x_T) + \regularizer(x_T) -
    f(x^*) - \regularizer(x^*)\right]} \nonumber \\ & \le L_T
\prox(x^*) + \frac{\extrastep_T}{\theta_T} \prox(x^*) +
\sum_{t=0}^{T-1} \frac{1}{2 \theta_t \extrastep_t} \dnorm{\error_t}^2
+ \sum_{t=0}^{T-1} \frac{1}{\theta_t} \<\error_t, z_t - x^*\> + T
\lipobj u.
    \label{eqn:main-result-no-expectations}
\end{align}
The non-probabilistic bound~\eqref{eqn:main-result-no-expectations} is the key
to the remainder of this proof, as well as the starting point for the proof of
Theorem~\ref{theorem:main-theorem-highprob-no-horizon} in the next section.
What remains is to take expectations in the
bound~\eqref{eqn:main-result-no-expectations}.  

Recall the filtration of $\sigma$-fields $\mc{F}_t$ so that $x_t, y_t,
z_t \in \mc{F}_{t-1}$, that is, $\mc{F}_t$ contains the randomness in the
stochastic oracle to time $t$. Since $g_t$ is an unbiased estimator of
$\nabla \smoothedfunc[t](y_t)$ by construction, we have $\E[g_t \mid \mc{F}_{t-1}] =
\nabla \smoothedfunc[t](y_t)$ and
\begin{equation*}
 \E[\<\error_t, z_t - x^*\>] = \E \big[ \E[\<\error_t, z_t - x^*\>
     \mid \mc{F}_{t-1}] \big] \; = \; \E \big[ \<\E[\error_t \mid
     \mc{F}_{t-1}], z_t - x^*\> \big] \;= \; 0,
\end{equation*}
where we have used the fact that $z_t$ is measurable with
respect to $\mc{F}_{t-1}$.  Now, recall from
Lemma~\ref{lemma:sequence-props} that $\theta_t \le \frac{2}{2 + t}$
and that $(1 - \theta_t) / \theta_t^2 = 1 / \theta_{t-1}^2$. Thus
\begin{equation*}
    \frac{\theta_{t-1}^2}{\theta_t^2} = \frac{1}{1 - \theta_t} \le
    \frac{1}{1 - \frac{2}{2 + t}} = \frac{2 + t}{t} \le \frac{3}{2}
    \qquad \mbox{for $t \ge 4$.}
  \end{equation*}
 Furthermore, we have $\theta_{t+1} \le \theta_t$, so by multiplying
 both sides of our bound~\eqref{eqn:main-result-no-expectations} by
 $\theta_{T-1}^2$ and taking expectations over the random vectors
 $g_t$,
\begin{align*}
  \lefteqn{\E[f(x_T) + \regularizer(x_T)] - [f(x^*) +
      \regularizer(x^*)]} \\ & \le \theta_{T-1}^2 L_T \prox(x^*) +
  \theta_{T-1}\extrastep_T \prox(x^*) + \theta_{T-1}\sum_{t=0}^{T-1}
  \frac{1}{2 \extrastep_t} \E\dnorm{\error_t}^2 + \theta_{T-1}
  \sum_{t=0}^{T-1} \E[\<\error_t, z_t - x^*\>] + \theta_{T-1}^2 T
  \lipobj \smoothparam \\ & \le \frac{6 \lipgrad \prox(x^*)}{T
    \smoothparam} + \frac{2 \extrastep_T \prox(x^*)}{T} +
  \frac{1}{T} \sum_{t=0}^{T-1} \frac{1}{\extrastep_t} \E
  \dnorm{\error_t}^2 + \frac{4 \lipobj \smoothparam}{T},
\end{align*}
where we used the fact that $L_T = \lipgrad / \smoothparam_T = \lipgrad /
\theta_T \smoothparam$. This completes the proof of
Theorem~\ref{theorem:main-theorem-no-horizon}. \\

As promised, we conclude with a theorem using a fixed setting of
the smoothing parameter $\smoothparam_t$.
\begin{theorem}
  \label{theorem:main-theorem}
  Suppose that $\smoothparam_t \equiv \smoothparam$ for all $t$ and set
  $L_t \equiv \lipgrad / \smoothparam$.  With the remaining conditions
  as in Theorem~\ref{theorem:main-theorem-no-horizon}, then for any $x^*
  \in \xdomain$, we have
  \begin{equation*}
    \E[f(x_T) + \regularizer(x_T)] - [f(x^*) + \regularizer(x^*)] \le
    \frac{4 \lipgrad \prox(x^*)}{T^2 \smoothparam} + \frac{2
      \extrastep_T \prox(x^*)}{T} + \frac{1}{T} \sum_{t=0}^{T-1}
    \frac{1}{\extrastep_t} \E\big[\dnorm{\error_t}^2\big] + \lipobj
    \smoothparam,
  \end{equation*}
where $\error_t \defn \nabla \smoothedfunc(y_t) - g_t$.
\end{theorem}
\begin{proof}
  The proof is brief. If we fix $\smoothparam_t \equiv \smoothparam$ for all
  $t$, then the bound~\eqref{eqn:main-result-no-expectations} holds with
  the last term $T \lipobj
  \smoothparam$ replaced by $\theta_{T-1}^2 \lipobj \smoothparam$, which we see
  by invoking Lemma~\ref{lemma:sequence-props}. The remainder of the proof
  follows unchanged, with $L_t \equiv \lipgrad$ for all $t$.
\end{proof}
\noindent
It is clear that by setting $u \propto 1/T$, the rates achieved by
Theorem~\ref{theorem:main-theorem-no-horizon} and
Theorem~\ref{theorem:main-theorem} are identical to constant factors.

\subsection{Proof of Theorem~\ref{theorem:main-theorem-highprob-no-horizon}}

An examination of the proof of Theorem~\ref{theorem:main-theorem-no-horizon}
shows that to control the probability of deviation from the expected
convergence rate, we need to control two terms: the squared error sequence
$\sum_{t=0}^{T-1} \frac{1}{2\extrastep_t} \dnorm{\error_t}^2$ and the sequence
$\sum_{t=0}^{T-1} \frac{1}{\theta_t} \<\error_t, z_t - x^*\>$. The next two
lemmas handle these terms.
\begin{lemma}
  \label{lemma:mg-error-terms}
  Let $\xdomain$ be compact with $\norm{x - x^*} \le \radius$ for all
  $x \in \xdomain$. Under Assumption~\ref{assumption:subexp-errors}, we have
  \begin{equation}
\label{EqnBanffErrA}
    \P \biggr[ \theta_{T-1}^2 \sum_{t=0}^{T-1} \frac{1}{\theta_t}
      \<\error_t, z_t - x^*\> \ge \epsilon \biggr] \le \exp
    \bigg(-\frac{T \epsilon^2}{\radius^2 \stddev^2} \bigg).
  \end{equation}
  Consequently, with probability at least $1 - \delta$,
  \begin{equation}
\label{EqnBanffErrB}
    \theta_{T-1}^2 \sum_{t=0}^{T-1} \frac{1}{\theta_t} \<\error_t, z_t - x^*\>
    \le \radius \stddev \sqrt{\frac{\log \frac{1}{\delta}}{T}}.
  \end{equation}
\end{lemma}
\begin{lemma}
  \label{lemma:opt-error2-terms}
  In the notation of
  Theorem~\ref{theorem:main-theorem-highprob-no-horizon} and under
  Assumption~\ref{assumption:subexp-errors}, we have
  \begin{equation}
   \label{EqnBanffTwoErrA}
    \log \P\biggr[\sum_{t=0}^{T-1} \frac{1}{2\extrastep_t}
    \dnorm{\error_t}^2 \ge \sum_{t=0}^{T-1} \frac{1}{2\extrastep_t}
    \E[\dnorm{\error_t}^2] + \epsilon\biggr] \leq
    \max\bigg\{-\frac{\epsilon^2}{ 32 e \stddev^4 \sum_{t=0}^{T-1}
      \frac{1}{\extrastep_t^2}}, -\frac{\extrastep_0}{4\stddev^2}
    \epsilon \bigg\}.
  \end{equation}
  Consequently, with probability at least $1 - \delta$,
  \begin{equation}
   \label{EqnBanffTwoErrB}
    \sum_{t=0}^{T-1} \frac{1}{2\extrastep_t} \dnorm{\error_t}^2
    \le \sum_{t=0}^{T-1} \frac{1}{2\extrastep_t} \E[\dnorm{\error_t}^2]
    + \frac{4 \stddev^2}{\extrastep} \max\bigg\{\log\frac{1}{\delta},
   \sqrt{2 e (\log T + 1) \log \frac{1}{\delta}}\bigg\}.
  \end{equation}
\end{lemma}
\noindent See Appendices~\ref{AppLemMGError} and~\ref{AppLemOpt2},
respectively, for the proofs of these two lemmas. \\

Equipped with these lemmas, we now prove
Theorem~\ref{theorem:main-theorem-highprob-no-horizon}.  Let us recall
the deterministic bound~\eqref{eqn:main-result-no-expectations} from
the proof of Theorem~\ref{theorem:main-theorem-no-horizon}:
\begin{align*}
    \lefteqn{\frac{1}{\theta_{T-1}^2}[f(x_T) + \regularizer(x_T) -
        f(x^*) - \regularizer(x^*)]} \\ & \le L_T \prox(x^*) +
    \frac{\extrastep_T}{\theta_T} \prox(x^*) + \sum_{t=0}^{T-1}
    \frac{1}{2 \theta_t \extrastep_t} \dnorm{\error_t}^2 +
    \sum_{t=0}^{T-1} \frac{1}{\theta_t} \<\error_t, z_t - x^*\> + T
    \lipobj u.
  \end{align*}
Noting that $\theta_{T-1} \le \theta_t$ for $t \in \{0, \ldots,
T-1\}$, Lemma~\ref{lemma:opt-error2-terms} implies that with
probability at least $1 - \delta$
\begin{equation*}
\theta_{T-1} \sum_{t=0}^{T-1} \frac{1}{2 \theta_t \extrastep_t}
\dnorm{\error_t}^2 \le \sum_{t=0}^{T-1} \frac{1}{2 \extrastep_t}
\E[\dnorm{\error_t}^2] + \frac{4 \stddev^2}{\extrastep}
\max\left\{\log(1/\delta), \sqrt{2e (\log T + 1)
  \log(1/\delta)}\right\}.
\end{equation*}
Applying Lemma~\ref{lemma:mg-error-terms}, we see that with
probability at least $1 - \delta$
\begin{equation*}
\theta_{T-1}^2 \sum_{t=0}^{T-1} \frac{1}{\theta_t} \<\error_t, z_t -
x^*\> \le \frac{\radius \stddev
  \sqrt{\log\frac{1}{\delta}}}{\sqrt{T}}.
\end{equation*}


The terms remaining to control are deterministic, and were bounded
previously in the proof of
Theorem~\ref{theorem:main-theorem-no-horizon}; in particular, we have
\begin{equation*}
\theta_{T-1}^2 L_T \le \frac{6 \lipgrad}{T \smoothparam}, ~~~
\frac{\theta_{T-1}^2 \extrastep_T}{\theta_T} \le \frac{4
  \extrastep_T}{T + 1}, ~~~ \mbox{and} ~~~ \theta_{T-1}^2 T \lipobj
\smoothparam \le \frac{4 \lipobj \smoothparam}{T + 1}.
\end{equation*}
Combining the above bounds completes the proof.

\section{Discussion}

In this paper, we have analyzed smoothing strategies for stochastic non-smooth
optimization. We have developed methods that are provably optimal in the
stochastic oracle model of optimization complexity, and given---to our
knowledge---the first variance reduction techniques for non-smooth stochastic
optimization. We think that at least two obvious questions remain. The first,
to which we have alluded earlier, is whether the randomized smoothing is
necessary to achieve such optimal rates of convergence. The second question is
whether dimension-independent smoothing techniques are possible, that is,
whether the $d$-dependent factors in the bounds in
Corollaries~\ref{corollary:variance-rate-ltwo}--\ref{corollary:highprob} are
necessary. Answering this question would require study of different smoothing
distributions, as the dimension dependence for our choices of $\smoothingdist$
is tight. We have outlined several applications for which smoothing techniques
give provable improvement over standard methods. Our experiments also show
qualitatively good agreement with the theoretical predictions we have
developed.

\subsection*{Acknowledgments}

We thank Alekh Agarwal and Ohad Shamir for discussions about variance
reduction in stochastic optimization and Steven Evans for some useful pointers
to smoothing operators. We also acknowledge the anonymous reviewers for their
helpful feedback and thoughtful comments. JCD was supported by a fellowship
from the National Defense Science and Engineering Graduate Fellowship Program
(NDSEG). PLB gratefully acknowledges the support of the NSF under award
DMS-0830410. In addition, MJW and JCD were partially supported by grant number
AFOSR-09NL184.

\appendix

\section{Proof of Lemma~\ref{lemma:lipschitz-sample-bound}}
\label{appendix:proof-of-lipschitz-concentration}

The proof of this lemma requires several auxiliary results on
sub-Gaussian and sub-exponential random variables, which we collect
and prove in Appendix~\ref{AppTail}. \\

Define $X_i = \nabla \smoothedfunc(x_t) - g_{i,t}$ and $S_m =
\sum_{i=1}^m X_i$, so  $\minv S_m = \nabla \smoothedfunc(x_t) -
\frac{1}{m} \sum_{i=1}^m g_{i,t}$. Note that conditioned on
$\mc{F}_{t-1}$, the $X_i$ are independent, so that for $L = 2
\lipobj$, we have $\dnorm{X_i} \le L$, and we can apply
Lemma~\ref{lemma:subgaussian-bounded-vectors} from
Appendix~\ref{AppTail}. In particular, we see that $\dnorm{\minv S_m}
- \E \dnorm{\minv S_m}$ is sub-Gaussian with parameter at most $4 L^2
/ m$. Consequently, we can apply Lemma~\ref{lemma:second-order-subg}
from Appendix~\ref{AppTail} so as to obtain
\begin{equation*}
\E \exp\left(\frac{s m \dnorm{\minv S_m}^2}{8 L^2}\right) \le
\frac{1}{\sqrt{1 - s}} \exp\left(\frac{m (\E\dnorm{\minv S_m})^2}{ 8
  L^2} \; \frac{s}{1 - s}\right).
\end{equation*}
Moreover, we can weaken the sub-Gaussian parameter $4 L^2 / m$ with
$\max\{\E \dnorm{\minv S_m}^2, 4L^2/m\}$:
\begin{equation*}
 \E \left[\exp(\lambda(\dnorm{S_m / m} - \E \dnorm{S_m / m}))\right]
 \le \exp\left(\frac{\lambda^2 \max\{4 L^2 / m, \E \dnorm{\minv
     S_m}^2\}}{2} \right).
\end{equation*} 
Recalling that for any random variable $X$, Jensen's inequality gives
$(\E X)^2 \le \E X^2$, we have
\begin{align*}
  \E \exp\left(\frac{s \dnorm{\minv S_m}^2}{ 2 \max\{\E \dnorm{\minv
      S_m}^2, \frac{4}{m} L^2\}}\right) & \le \frac{1}{\sqrt{1 - s}}
  \exp\left(\frac{\E \dnorm{\minv S_m}^2}{2 \max\{\E \dnorm{\minv
      S_m}^2, \frac{4}{m} L^2\}} \; \frac{s}{1 - s}\right) \\ & \le
  \frac{1}{\sqrt{1 - s}} \exp\left(\half \cdot \frac{s}{1 - s}\right).
\end{align*}
By taking $s = \half$, we get
\begin{equation*}
  \frac{1}{\sqrt{1 - s}} \exp\left(\half \; \frac{s}{1 - s}\right) =
  \sqrt{2} \exp\left(\half\right) \le \exp(1).
\end{equation*}
Replacing $L$ with $2 \lipobj$ completes the proof.


\section{Proof of Lemma~\ref{lemma:variance-sequence-rda}}
\label{AppLemVarRDA}

Define the linearized version of the cumulative objective
\begin{equation}
  \label{eqn:cumulative-linear-loss}
  \ell_t(z) \defeq
  \sum_{\tau = 0}^t \frac{1}{\theta_\tau} [\smoothedfunc[\tau](y_\tau) + \<g_\tau, z
    - y_\tau\> + \regularizer(z)],
\end{equation}
and let $\ell_{-1}(z)$ denote the indicator function of the set $\xdomain$. For
conciseness, we adopt the shorthand
\begin{align*}
  \stepsize_t^{-1} & =  L_t + \extrastep_t / \theta_t
  \qquad \mbox{and} \qquad
  \phi_t(x) = \smoothedfunc[t](x) + \regularizer(x).
\end{align*}
By the smoothness of $\smoothedfunc[t]$, we have
\begin{align*}
  \underbrace{\smoothedfunc[t](x_{t+1}) + \regularizer(x_{t+1})}_{\phi_t(x_{t+1})} &
  \leq \smoothedfunc[t](y_t) + \<\nabla \smoothedfunc[t](y_t), x_{t+1} - y_t\> + \frac{L_t}{2}
  \norm{x_{t+1} - y_t}^2 + \regularizer(x_{t+1}).
\end{align*}
From the definition~\eqref{eqn:y-update}--\eqref{eqn:x-update} of the triple
$(x_t, y_t, z_t)$, we obtain
\begin{align*}
  \phi_t(x_{t+1}) & \leq \smoothedfunc[t](y_t) + \<\nabla \smoothedfunc[t](y_t), \theta_t z_{t+1} +
  (1 - \theta_t) x_t\> + \frac{L_t}{2}\norm{\theta_t z_{t + 1} -
    \theta_t z_t}^2 + \regularizer(\theta_t z_{t+1} + (1 - \theta_t)
  x_t).
\end{align*}
Finally, by convexity of the regularizer $\regularizer$, we conclude 
\begin{align}
  \phi_t(x_{t+1}) & \leq 
  \theta_t \left[\smoothedfunc[t](y_t) + \<\nabla \smoothedfunc[t](y_t), z_{t+1} - y_t\> +
    \regularizer(z_{t+1}) + \frac{L_t \theta_t}{2} \norm{z_{t+1} -
      z_t}^2 \right] \nonumber \\
  & \qquad\qquad\qquad
  ~ + (1 - \theta_t)[\smoothedfunc[t](y_t) + \<\nabla \smoothedfunc[t](y_t), x_t - y_t\> +
    \regularizer(x_t)].
  \label{EqnBanffPreLunch}
\end{align}

By the strong convexity of $\prox$, it is clear that we have the lower bound
\begin{align}
  \label{EqnDivLower}
  \divergence(x,y) \; = \; \prox(x) - \prox(y)
  - \<\nabla \prox(y), x - y\> \; \geq \; \half \, \norm{x-y}^2.
\end{align}
On the other hand, by the convexity of $\smoothedfunc[t]$, we have
\begin{align}
  \label{EqnFbound}
  \smoothedfunc[t](y_t) + \<\nabla \smoothedfunc[t](y_t), x_t - y_t\>  & \leq \smoothedfunc[t](x_t).
\end{align}
Substituting inequalities~\eqref{EqnDivLower} and~\eqref{EqnFbound}
into the upper bound~\eqref{EqnBanffPreLunch} and simplifying yields
\begin{align*}
  \phi_t(x_{t+1}) & \leq \theta_t\left[\smoothedfunc[t](y_t) + \<\nabla \smoothedfunc[t](y_t),
  z_{t+1} - y_t\> + \regularizer(z_{t+1}) + L_t\theta_t
  \divergence(z_{t+1}, z_t)\right] + (1 - \theta_t) [\smoothedfunc[t](x_t) +
  \regularizer(x_t)].
\end{align*}

We now re-write this upper bound in terms of the error
$\error_t = \nabla \smoothedfunc[t](y_t) - g_t$.  In particular,
\begin{align}
  \lefteqn{\phi_t(x_{t+1})} \nonumber \\
  & \le \theta_t \left[\smoothedfunc[t](y_t) + \<g_t, z_{t+1} - y_t\> +
    \regularizer(z_{t+1}) + L_t \theta_t \divergence(z_{t+1},
    z_t)\right] \nonumber \\
  & \qquad\quad~ + (1 - \theta_t) [\smoothedfunc[t](x_t) + \regularizer(x_t)] +
  \theta_t \<\error_t, z_{t + 1} - y_t\> \nonumber \\ & = \theta_t^2
  \left[\ell_t(z_{t+1}) - \ell_{t-1}(z_{t+1}) + L_t \divergence(z_{t+1},
    z_t)\right] + (1 - \theta_t) [\smoothedfunc[t](x_t) + \regularizer(x_t)] +
  \theta_t \<\error_t, z_{t+1} - y_t\>.
  \label{eqn:intermediate-rda-bound-adaptive}
\end{align}

Using the fact that $z_t$ minimizes $\ell_{t-1}(x) +
\frac{1}{\stepsize_t} \prox(x)$, the first order conditions for
optimality imply that for all $g \in \partial \ell_{t-1}(z_t)$, we
have $\<g + \frac{1}{\stepsize_t} \nabla \prox(z_t), x - z_t\> \ge
0$. Thus, first-order convexity gives
\begin{align*}
  \ell_{t-1}(x) - \ell_{t-1}(z_t) & \ge \<g, x - z_t \> \ge
  -\frac{1}{\stepsize_t} \<\nabla \prox(z_t), x - z_t\> =
  \frac{1}{\stepsize_t} \prox(z_t) - \frac{1}{\stepsize} \prox(x) +
  \frac{1}{\stepsize_t} \divergence(x, z_t).
\end{align*}
Adding $\ell_t(z_{t + 1})$ to both sides of the above and
substituting $x = z_{t+1}$, we conclude
  \begin{equation*}
    \ell_t(z_{t+1}) - \ell_{t-1}(z_{t+1})
    \le \ell_t(z_{t+1}) - \ell_{t-1}(z_t) - \frac{1}{\stepsize_t} \prox(z_t)
    + \frac{1}{\stepsize_t} \prox(z_{t+1}) - \frac{1}{\stepsize_t}
    \divergence(z_{t+1}, z_t).
  \end{equation*}
Combining this inequality with the
bound~\eqref{eqn:intermediate-rda-bound-adaptive} and using the
definition $\stepsize_t^{-1} = L_t + \extrastep_t / \theta_t$, we find
\begin{align*}
  \smoothedfunc[t](x_{t+1}) + \regularizer(x_{t+1}) & \le
  \theta_t^2\left[\ell_t(z_{t+1}) - \ell_t(z_t) - \frac{1}{\stepsize_t}
  \prox(z_t) + \frac{1}{\stepsize_t} \prox(z_{t+1}) -
  \frac{\extrastep_t}{\theta_t} \divergence(z_{t+1}, z_t)\right] \\
  & \qquad ~ + (1 - \theta_t)[\smoothedfunc[t](x_t) + \regularizer(x_t)] + \theta_t
  \<\error_t, z_{t+1} - y_t\> \\
  & \le \theta_t^2\left[\ell_t(z_{t+1}) -
    \ell_t(z_t) - \frac{1}{\stepsize_t} \prox(z_t) +
    \frac{1}{\stepsize_{t+1}} \prox(z_{t+1}) -
    \frac{\extrastep_t}{\theta_t} \divergence(z_{t+1}, z_t)\right] \\
  & \qquad ~ + (1 - \theta_t)[\smoothedfunc[t](x_t) + \regularizer(x_t)] + \theta_t
  \<\error_t, z_{t+1} - y_t\>
\end{align*}
since $\stepsize_t^{-1}$ is non-decreasing.  We now divide both
sides by $\theta_t^2$, and unwrap the recursion.  Recall that $(1 -
\theta_t) / \theta_t^2 = 1/\theta_{t-1}^2$ and $\smoothedfunc[t] \le \smoothedfunc[t-1]$ by
construction, so we obtain
\begin{align*}
  \frac{1}{\theta_t^2}[\smoothedfunc[t](x_{t+1}) + \regularizer(x_{t+1})]
  & \le \frac{1 - \theta_t}{\theta_t^2} [\smoothedfunc[t](x_t) + \regularizer(x_t)]
  + \ell_t(z_{t+1}) - \ell_t(z_t) - \frac{1}{\stepsize_t} \prox(z_t)
  + \frac{1}{\stepsize_{t+1}} \prox(z_{t+1}) \\
  & \qquad ~ - \frac{\extrastep_t}{\theta_t}
  \divergence(z_{t+1}, z_t) + \frac{1}{\theta_t} \<e_t, z_{t + 1} - y_t\> \\
  & \stackrel{(i)}{=} \frac{1}{\theta_{t-1}^2} [\smoothedfunc[t](x_t) + \regularizer(x_t)]
  + \ell_t(z_{t+1}) - \ell_t(z_t) - \frac{1}{\stepsize_t} \prox(z_t)
  + \frac{1}{\stepsize_{t+1}} \prox(z_{t+1}) \\
  & \qquad ~ - \frac{\extrastep_t}{\theta_t}
  \divergence(z_{t+1}, z_t) + \frac{1}{\theta_t} \<e_t, z_{t + 1} - y_t\> \\
  & \stackrel{(ii)}{\le}
  \frac{1}{\theta_{t-1}^2} [\smoothedfunc[t-1](x_t) + \regularizer(x_t)]
  + \ell_t(z_{t+1}) - \ell_t(z_t) - \frac{1}{\stepsize_t} \prox(z_t)
  + \frac{1}{\stepsize_{t+1}} \prox(z_{t+1}) \\
  & \qquad ~ - \frac{\extrastep_t}{\theta_t}
  \divergence(z_{t+1}, z_t) + \frac{1}{\theta_t} \<e_t, z_{t + 1} - y_t\>.
\end{align*}
The equality (i) follows since $(1 - \theta_t) / \theta_t^2 = 1 /
\theta_{t-1}^2$, while the inequality (ii) is a consequence of the fact that
$\smoothedfunc[t] \le \smoothedfunc[t-1]$. By applying the three steps above successively to
$[\smoothedfunc[t-1](x_t) + \regularizer(x_t)] / \theta_{t-1}^2$, then to
$[\smoothedfunc[t-2](x_{t-1}) + \regularizer(x_{t-1})] / \theta_{t-2}^2$, and so on until
$t = 0$, we find
\begin{align*}
  \frac{1}{\theta_t^2}[\smoothedfunc[t](x_{t+1}) + \regularizer(x_{t+1})]
  & \le \frac{1 - \theta_0}{\theta_0^2} [\smoothedfunc[0](x_0)
    + \regularizer(x_0)]
  + \ell_t(z_{t+1}) + \frac{1}{\stepsize_{t+1}} \prox(z_{t+1})
  - \frac{1}{\stepsize_0} \prox(z_0) \\
  & \qquad ~ -\sum_{\tau=0}^t \frac{\extrastep_\tau}{\theta_\tau}
  \divergence(z_{\tau + 1}, z_\tau)
  + \sum_{\tau=0}^t \frac{1}{\theta_\tau} \<\error_\tau, z_{\tau+1} -
  y_\tau\>
  - \ell_{-1}(z_0).
\end{align*}

By construction, $\theta_0 = 1$, we have $\ell_{-1}(z_0) = 0$, and
$z_{t+1}$ minimizes $\ell_t(x) + \frac{1}{\stepsize_{t+1}} \prox(x)$
over $\xdomain$. Thus, for any $x^* \in \xdomain$, we have
\begin{equation*}
  \frac{1}{\theta_t^2}[\smoothedfunc[t](x_{t + 1}) + \regularizer(x_{t+1})]
  \le \ell_t(x^*) + \frac{1}{\stepsize_{t+1}}\prox(x^*)
  - \sum_{\tau=0}^t \frac{\extrastep_\tau}{\theta_\tau}
  \divergence(z_{\tau + 1}, z_\tau) + \sum_{\tau=0}^t \frac{1}{\theta_\tau}
  \<\error_\tau, z_{\tau + 1} - y_\tau\>.
\end{equation*}
Recalling the definition~\eqref{eqn:cumulative-linear-loss} of $\ell_t$ and
noting that $\smoothedfunc[t](y_t) + \<\nabla \smoothedfunc[t](y_t), x - y_t\> \le \smoothedfunc[t](x)$ by
convexity, we expand $\ell_t$ and have
\begin{align}
  \lefteqn{\frac{1}{\theta_t^2}[\smoothedfunc[t](x_{t+1}) + \regularizer(x_{t+1})]}
  \nonumber \\
  & \le \sum_{\tau=0}^t \frac{1}{\theta_\tau}
  [\smoothedfunc[\tau](y_\tau) + \<g_\tau, x^* - y_\tau\> + \regularizer(x^*)]
  + \frac{1}{\stepsize_{t+1}} \prox(x^*)
  - \sum_{\tau=0}^t \frac{\extrastep_\tau}{\theta_\tau}
  \divergence(z_{\tau + 1}, z_\tau) + \sum_{\tau=0}^t \frac{1}{\theta_\tau}
  \<\error_\tau, z_{\tau + 1} - y_t\> \nonumber \\
  & = \sum_{\tau=0}^t \frac{1}{\theta_\tau}
  [\smoothedfunc[\tau](y_\tau) + \<\nabla \smoothedfunc[\tau](y_\tau), x^* - y_\tau\> +
    \regularizer(x^*)]
  + \frac{1}{\stepsize_{t+1}} \prox(x^*)
  - \sum_{\tau=0}^t \frac{\extrastep_\tau}{\theta_\tau}
  \divergence(z_{\tau + 1}, z_\tau) + \sum_{\tau=0}^t \frac{1}{\theta_\tau}
  \<\error_\tau, z_{\tau + 1} - x^*\> \nonumber \\
  & \le \sum_{\tau=0}^t \frac{1}{\theta_\tau}[
    \smoothedfunc[\tau](x^*) + \regularizer(x^*)]
  + \frac{1}{\stepsize_{t+1}} \prox(x^*)
  - \sum_{\tau=0}^t \frac{\extrastep_\tau}{\theta_\tau}
  \divergence(z_{\tau + 1}, z_\tau) + \sum_{\tau=0}^t \frac{1}{\theta_\tau}
  \<\error_\tau, z_{\tau + 1} - x^*\>.
  \label{eqn:rda-to-fenchel-adaptive}
\end{align}
Now we use the Fenchel-Young inequality applied to the conjugates
$\half \norm{\cdot}^2$ and $\half\dnorm{\cdot}^2$, which gives
\begin{equation*}
  \<\error_t, z_{t+1} - x^*\>
  = \<\error_t, z_t - x^*\> + \<\error_t, z_{t + 1} - z_t\>
  \le \<\error_t, z_t - x^*\>
  + \frac{1}{2 \extrastep_t} \dnorm{\error_t}^2
  + \frac{\extrastep_t}{2} \norm{z_t - z_{t + 1}}^2.
\end{equation*}
In particular,
\begin{equation*}
  -\frac{\extrastep_t}{\theta_t} \divergence(z_{t+1}, z_t)
  + \frac{1}{\theta_t} \<\error_t, z_{t+1} - x^*\>
  \le \frac{1}{2 \extrastep_t \theta_t} \dnorm{\error_t}^2
  + \frac{1}{\theta_t} \<\error_t, z_t - x^*\>.
\end{equation*}
Using this inequality and rearranging~\eqref{eqn:rda-to-fenchel-adaptive}
gives the statement of the lemma.


\section{Proof of Lemma~\ref{lemma:mg-error-terms}}
\label{AppLemMGError}

Consider the sequence $\sum_{t=0}^{T-1} \frac{1}{\theta_t} \<\error_t,
z_t - x^*\>$. Since $\xdomain$ is compact and $\norm{z_t - x^*} \le
\radius$, we have $\<\error_t, z_t - x^*\> \le \dnorm{\error_t}
\radius$. Further, $\E[\<\error_t, z_t - x^*\> \mid \mc{F}_{t-1}] =
0$, so that $\frac{1}{\theta_t}\<\error_t, z_t - x^*\>$ is a
martingale difference sequence. Further, by setting $c_t = \radius
\stddev / \theta_t$, we have
\begin{equation*}
 \E\bigg[\exp\bigg(\frac{\<\error_t, z_t - x^*\>^2}{c_t^2
   \theta_t^2}\bigg) \mid \mc{F}_{t-1} \bigg]
 \le \E\bigg[\exp\left(\frac{\dnorm{\error_t}^2
   \radius^2}{c_t^2 \theta_t^2}\right) \mid \mc{F}_{t-1}\bigg]
 = \E\bigg[\exp\bigg(\frac{\dnorm{\error_t}^2}{\stddev^2}\bigg)
   \mid \mc{F}_{t-1}\bigg]
 \le \exp(1)
\end{equation*}
by Assumption~\ref{assumption:subexp-errors}.  By applying
Lemma~\ref{lemma:subexp-squared-implies-subg} from
Appendix~\ref{AppTail}, we conclude that $\frac{1}{\theta_t}
\<\error_t, z_t - x^*\>$ is (conditionally) sub-Gaussian with
parameter $\sigma_t^2 \le 4 \radius^2 \stddev^2 / 3
\theta_t^2$. Applying the Azuma-Hoeffding inequality (see
Eq.~\eqref{eqn:hoeffding-azuma}, Appendix~\ref{AppTail}) yields
\begin{equation*}
  \P \biggr [ \sum_{t=0}^{T-1} \frac{1}{\theta_t} \<\error_t, z_t -
    x^*\> \ge w \biggr] \leq \exp \biggr ( -\frac{3 w^2}{8
    \radius^2\stddev^2 \sum_{t=0}^{T-1} \frac{1}{\theta_t^2}} \biggr).
\end{equation*}
Setting $w = \epsilon/\theta_{T-1}$ yields that
\begin{equation*}
  \P \biggr[\theta_{T-1} \sum_{t=0}^{T-1} \frac{1}{\theta_t} \<\error_t,
    z_t - x^*\> \ge \epsilon\biggr]
  \le \exp \biggr(\!-\frac{3 \epsilon^2}{8
  \radius^2 \stddev^2 \sum_{t=0}^{T-1}
  \frac{\theta_{T-1}^2}{\theta_t^2}} \biggr).
\end{equation*}
Noting that $\theta_{T-1} \le \theta_t$ for any $t < T$, we have $\radius^2
\stddev^2 \sum_{t=0}^{T-1} \frac{\theta_{T-1}^2}{\theta_t^2} \le \radius^2
\stddev^2 \sum_{t=0}^{T-1} 1 = \radius^2 \stddev^2 T$, dividing $\epsilon$
again by $\theta_{T-1}$, and recalling that $\theta_{T-1} \le \frac{2}{T +
  1}$, we have
\begin{equation*}
 \P \left[\theta_{T-1}^2 \sum_{t=0}^{T-1} \frac{1}{\theta_t}
   \<\error_t, z_t - x^*\> \ge \epsilon\right]
 \le \exp\left(-\frac{12(T
   + 1)\epsilon^2}{8 \radius^2 \stddev^2}\right) \le
 \exp\left(-\frac{3 T \epsilon^2}{2 \radius^2 \stddev^2}\right),
  \end{equation*}
as claimed~\eqref{EqnBanffErrA}.  The second
claim~\eqref{EqnBanffErrB} follows by setting $\delta = \exp(-\frac{3 T
  \epsilon^2}{2 \radius^2 \stddev^2})$, and then solving to obtain
$\epsilon^2 = \frac{2 \radius^2 \stddev^2}{3 T} \log\frac{1}{\delta}$.


\section{Proof of Lemma~\ref{lemma:opt-error2-terms}}
\label{AppLemOpt2}

Again, recall the $\sigma$-fields $\mathcal{F}_t$ defined prior to
Assumption~\ref{assumption:subexp-errors}.  Define the random variables 
\begin{align*}
X_t & \defeq \frac{1}{2 \extrastep_t} \dnorm{\error_t}^2 -
\frac{1}{2\extrastep_t} \E[\dnorm{\error_t}^2 \mid \mc{F}_{t-1}].
\end{align*}
As an intermediate step, we claim that for $\lambda \le \extrastep_t /
2\stddev^2$, the following bound holds:
\begin{equation}
  \label{EqnIntermediate}
  \E[\exp(\lambda X_t) \mid \mc{F}_{t-1}] =
  \E\left[\exp\left(\frac{\lambda}{2\extrastep_t} (\dnorm{\error_t}^2 -
    \E[\dnorm{\error_t}^2 \mid \mc{F}_{t-1}]) \right) \mid
    \mc{F}_{t-1}\right] \le \exp\left(\frac{8 e}{\extrastep_t^2}
  \lambda^2 \stddev^4\right).
\end{equation}
For now, we proceed with the proof, returning to establish this intermediate
claim later.

The bound~\eqref{EqnIntermediate} implies that $X_t$ is
sub-exponential with parameters $\subexpbound_t =
\extrastep_t/2\stddev^2$ and \mbox{$\subexpparam_t^2 \le 16 e
  \stddev^4 / \extrastep_t^2$.}  Since $\extrastep_t = \extrastep
\sqrt{t + 1}$, it is clear that $\min_t\{\subexpbound_t\} =
\subexpbound_0 = \extrastep_0 / 2\stddev^2$.  By defining $C^2 =
\sum_{t=0}^{T-1} \subexpparam_t^2$, we can apply Theorem I.5.1 from
the book~\cite[pg.~26]{BuldyginKo00} to conclude that
\begin{equation}
\label{eqn:subexp-concentration}
\P\left(\sum_{t=0}^{T-1} X_t \ge \epsilon\right) \leq
\begin{cases} 
 \exp\left(-\frac{\epsilon^2}{2C^2}\right) & \mbox{for}~ 0 \le
 \epsilon \le \subexpbound_0C^2 \\ \exp\left(-\frac{\subexpbound_0
   \epsilon}{2}\right) & \mbox{otherwise, i.e.}~ \epsilon >
 \subexpbound_0 C^2,
\end{cases}
  \end{equation}
which yields the first claim in Lemma~\ref{lemma:opt-error2-terms}.

The second statement involves inverting the bound for the different
regimes of $\epsilon$. Before proving the bound, we note that for
$\epsilon = \subexpbound_0 C^2$, we have $\exp(-\epsilon^2 / 2C^2) =
\exp(-\subexpbound \epsilon / 2)$, so we can invert each of the $\exp$
terms to solve for $\epsilon$ and take the maximum of the bounds.  We
begin with $\epsilon$ in the regime closest to zero, recalling that
$\extrastep_t = \extrastep \sqrt{t + 1}$. We see that
  \begin{equation*}
    C^2 \le \frac{16 e \stddev^4}{\extrastep^2} \sum_{t=0}^{T-1}
    \frac{1}{t + 1} \le \frac{16 e \stddev^4}{\extrastep^2} \log(T + 1).
  \end{equation*}
  Thus, inverting the bound $\delta = \exp(-\epsilon^2 / 2C^2)$, we get
  $\epsilon = \sqrt{2 C^2 \log \frac{1}{\delta}}$, or that
  \begin{equation*}
    \sum_{t=0}^{T-1} \frac{1}{2\extrastep_t} \dnorm{\error_t}^2
    \le \sum_{t=0}^{T-1} \frac{1}{2\extrastep_t} \E[\dnorm{\error_t}]^2
    + 4\sqrt{2 e} \frac{\stddev^2}{\extrastep}
    \sqrt{\log\frac{1}{\delta} \log(T + 1)}
  \end{equation*}
  with probability at least $1 - \delta$. In the large $\epsilon$ regime,
  we solve $\delta = \exp(-\extrastep \epsilon / 4\stddev^2)$ or
  $\epsilon = \frac{4 \stddev^2}{\extrastep}\log\frac{1}{\delta}$, which
  gives that
  \begin{equation*}
    \sum_{t=0}^{T-1} \frac{1}{2 \extrastep_t} \dnorm{\error_t}^2
    \le \sum_{t=0}^{T-1} \frac{1}{2\extrastep_t} \E\dnorm{\error_t}^2
    + \frac{4 \stddev^2}{\extrastep} \log \frac{1}{\delta}
  \end{equation*}
  with probability at least $1 - \delta$, by the
  bound~\eqref{eqn:subexp-concentration}.

We now return to prove the intermediate claim~\eqref{EqnIntermediate}.
Let $X \defeq \dnorm{\error_t}$.  By assumption, we have $\E\exp(X^2 /
\stddev^2) \le \exp(1)$, so for $\lambda \in [0,1]$ we see
  \begin{equation*}
    \P(X^2 / \stddev^2 > \epsilon) \le
    \E[\exp(\lambda X^2 / \stddev^2)] \exp(-\lambda \epsilon)
    \le \exp(\lambda - \lambda \epsilon)
  \end{equation*}
  and replacing $\epsilon$ with $1 + \epsilon$ we have
  $\P(X^2 > \stddev^2 + \epsilon \stddev^2) \le \exp(-\epsilon)$.
  If $\epsilon\stddev^2 \ge \stddev^2 - \E X^2$, then $\stddev^2 - \E X^2
  + \epsilon \stddev^2 \le 2\epsilon \stddev^2$ so
  \begin{equation*}
    \P(X^2 > \E X^2 + 2 \epsilon \stddev^2)
    \le \P(X^2 > \stddev^2 + \epsilon \stddev^2)
    \le \exp(-\epsilon),
  \end{equation*}
  while for $\epsilon \stddev^2 < \stddev^2 - \E X^2$, we clearly have $\P(X^2
  - \E X^2 > \epsilon \stddev^2) \le 1 \le \exp(1) \exp(-\epsilon)$ since
  $\epsilon \le 1$.  In either case, we have
  \begin{equation*}
    \P(X^2 - \E X^2 > \epsilon) \le \exp(1)
    \exp\left(-\frac{\epsilon}{2 \stddev^2}\right).
  \end{equation*}
  For the opposite concentration inequality, we see that
  \begin{equation*}
    \P((\E X^2 - X^2) / \stddev^2 > \epsilon)
    \le \E[\exp(\lambda \E X^2 / \stddev^2)\exp(-\lambda X^2 / \stddev^2)]
    \exp(-\lambda \epsilon)
    \le \exp(\lambda - \lambda \epsilon)
  \end{equation*}
  or $\P(X^2 - \E X^2 < -\stddev^2 \epsilon) \le \exp(1)\exp(-\epsilon)$.
  Using the union bound, we have
  \begin{equation}
    \label{eqn:concentration-error2}
    \P(|X^2 - \E X^2| > \epsilon) \le 2\exp(1)
    \exp\left(-\frac{\epsilon}{ 2 \stddev^2}\right).
  \end{equation}

  Now we apply Lemma~\ref{lemma:concentration-to-subexp} to the
  bound~\eqref{eqn:concentration-error2} to see that
  $\dnorm{\error_t}^2 - \E[\dnorm{\error_t}^2] = X^2 - \E[X^2]$ is
  sub-exponential with parameters $\subexpbound \ge \stddev^2$ and
  $\subexpparam^2 \le 32 e \stddev^4$.


\section{Properties of randomized smoothing}
\label{appendix:smoothing}

In this section, we discuss the analytic properties of the smoothed
function $\smoothedfunc$ from the convolution~\eqref{eqn:convolve}.
We assume throughout that functions are sufficiently integrable
without bothering with measurability conditions (since $F(\cdot;
\statsample)$ is convex, this is no real loss of
generality~\cite{Bertsekas73,RockafellarWe98}). By Fubini's theorem,
we have
\begin{equation*}
  \smoothedfunc(x) = \int_{\R^d} \int_\statsamplespace F(x + y; \statsample)
  d\statprob(\statsample)
  \smoothingdist(y) dy
  = \int_\statsamplespace \int_{\R^d} F(x + y; \statsample) \smoothingdist(y)
  dy d\statprob(\statsample)
  = \int_\statsamplespace F_\smoothingdist(x; \statsample)
  d\statprob(\statsample).
\end{equation*}
Here $F_\smoothingdist(x; \statsample) = (F(\cdot; \statsample) *
\smoothingdist(-\cdot))(x)$. We begin with the observation that since
$\smoothingdist$ is a density with respect to Lebesgue measure, the
function $\smoothedfunc$ is in fact
differentiable~\cite{Bertsekas73}. So we have already made our problem
somewhat smoother, as it is now differentiable; for the remainder, we
consider finer properties of the smoothing operation. In particular,
we will show that under suitable conditions on $\smoothingdist$,
$F(\cdot; \statsample)$, and $\statprob$, the function
$\smoothedfunc$ is uniformly close to $f$ over $\xdomain$ and
$\nabla \smoothedfunc$ is Lipschitz continuous.

\subsection{Statements of smoothing lemmas}

A remark on notation before proceeding: since $f$ is convex, it is
almost-everywhere differentiable, and we can abuse notation and take its
gradient inside of integrals and expectations with respect to Lebesgue
measure. Similarly, $F(\cdot; \statsample)$ is almost everywhere
differentiable with respect to Lebesgue measure, so we use the same abuse of
notation for $F$ and write $\nabla F(x + Z; \statsample)$, which exists with
probability 1.  We prove the following set of smoothness lemmas at the end of
this section.

\begin{lemma}
  \label{lemma:linf}
  Let $\smoothingdist$ be the uniform density on the
  $\ell_\infty$-ball of radius $\smoothparam$. Assume that
  \mbox{$\E[\linf{\partial F(x; \statsample)}^2] \le \lipobj^2$} for
  all $x \in \xdomain + B_\infty(0, u)$ Then
  \begin{enumerate}[(i)]
  \item $f(x) \le \smoothedfunc(x) \le f(x) + \frac{\lipobj d}{2} u$
  \item $\smoothedfunc$ is $\lipobj$-Lipschitz with respect to the
    $\ell_1$-norm over $\xdomain$.
  \item $\smoothedfunc$ is continuously differentiable; moreover, its
    gradient is $\frac{\lipobj}{\smoothparam}$-Lipschitz continuous with
    respect to the $\ell_1$-norm.
  \item Let $Z \sim \smoothingdist$. Then $\E[\nabla F(x + Z; \statsample)] =
    \nabla \smoothedfunc(x)$ and $\E[\linf{\nabla \smoothedfunc(x) -
      \nabla F(x + Z; \statsample)}^2] \le 4 \lipobj^2$.
  \end{enumerate}
  There exists a function $f$ for which each of the estimates (i)--(iii) are
  tight simultaneously, and (iv) is tight
  at least to a factor of $1/4$.
\end{lemma} 
\paragraph{Remark:}  Note that
the hypothesis of this lemma is satisfied if for any fixed
$\statsample \in \statsamplespace$, the function $F(\cdot;
\statsample)$ is $\lipobj$-Lipschitz with respect to the
$\ell_1$-norm.

\vspace*{.2in}

The following lemma provides bounds for uniform smoothing of functions
Lipschitz with respect to the $\ell_2$-norm while sampling from an
$\ell_\infty$-ball. 
\begin{lemma}
  \label{lemma:linf-to-ltwo}
  Let $\smoothingdist$ be the uniform density on $B_\infty(0,
  \smoothparam)$ and assume  that $\E[\ltwo{\partial F(x;
      \statsample)}^2] \le \lipobj^2$ for $x \in \xdomain +
  B_\infty(0, u)$. Then
  \begin{enumerate}[(i)]
  \item The function $f$ satisfies the upper bound $f(x) \le
    \smoothedfunc(x) \le f(x) + \lipobj u \sqrt{d}$
  \item The function $\smoothedfunc$ is $\lipobj$-Lipschitz over
    $\xdomain$.
  \item The function $\smoothedfunc$ is continuously
    differentiable; moreover, its gradient is $\frac{2 \sqrt{d}
      \lipobj}{u}$ Lipschitz continuous.
  \item For random variables $Z \sim \smoothingdist$ and $\statsample
    \sim \statprob$, we have 
    \begin{equation*}
      \E[\nabla F(x + Z; \statsample)] = \nabla \smoothedfunc(x), \quad
      \mbox{and} \quad \E[ \ltwo{\nabla \smoothedfunc(x) - \nabla F(x +
          Z; \statsample)}^2] \le \lipobj^2.
    \end{equation*}
    The latter estimate is tight.
  \end{enumerate}
\end{lemma}

\vspace*{.2in}

A similar lemma can be proved when $\smoothingdist$ is the density of
the uniform distribution on $B_2(0, u)$. In this case, Yousefian et al.\ give
(i)--(iii) of the following lemma~\cite{YousefianNeSh12}.
\begin{lemma}[Yousefian, Nedi\'{c}, Shanbhag]
  \label{lemma:ltwo}
  Let $\smoothedfunc$ be defined as in \eqref{eqn:convolve} where
  $\smoothingdist$ is the uniform density on the $\ell_2$-ball of radius
  $u$. Assume that $\E[\ltwo{\partial F(x; \statsample)}^2] \le \lipobj^2$ for
  $x \in \xdomain + B_2(0, u)$. Then
  \begin{enumerate}[(i)]
  \item $f(x) \le \smoothedfunc(x) \le f(x) + \lipobj u$
  \item $\smoothedfunc$ is $\lipobj$-Lipschitz over $\xdomain$.
  \item $\smoothedfunc$ is continuously differentiable; moreover,
    its gradient is $\frac{\lipobj \sqrt{d}}{u}$-Lipschitz continuous.
  \item Let $Z \sim \smoothingdist$. Then $\E[\nabla F(x + Z; \statsample)] =
    \nabla \smoothedfunc(x)$, and $\E[\ltwo{\nabla \smoothedfunc(x) -
      \nabla F(x + Z; \statsample)}^2] \le \lipobj^2$.
  \end{enumerate}
  In addition, there exists a function $f$ for which each of the bounds
  (i)--(iv) is tight---cannot be improved by more than a constant
  factor---simultaneously.
\end{lemma}
Lastly, for situations in which $F(\cdot; \statsample)$ is $\lipobj$-Lipschitz
with respect to the $\ell_2$-norm over all of $\R^d$ and for
$\statprob$-a.e.\ $\statsample$, we can use the normal distribution to perform
smoothing of the expected function $f$. The following lemma is similar to a
result of Lakshmanan and de Farias~\cite[Lemma 3.3]{LakshmananFa08}, but they
consider functions Lipschitz-continuous with respect to the
$\ell_\infty$-norm, i.e.\ $|f(x) - f(y)| \le L \linf{x - y}$, which is too
stringent for our purposes, and we carefully quantify the dependence on the
dimension of the underlying problem.
\begin{lemma}
  \label{lemma:normal}
  Let $\smoothingdist$ be the $N(0, u^2 I_{d \times d})$ distribution.
  Assume that $F(\cdot; \statsample)$ is $\lipobj$-Lipschitz with
  respect to the $\ell_2$-norm---that is
  \begin{align*}
    \sup\{ \ltwo{g} \mid g \in \partial F(x; \statsample), x \in
    \xdomain\} & \le \lipobj \qquad \mbox{for
      $\statprob$-a.e.\ $\statsample$.}
  \end{align*}
  Then the following properties hold:
  \begin{enumerate}[(i)]
  \item $f(x) \le \smoothedfunc(x) \le f(x) + \lipobj u \sqrt{d}$
  \item $\smoothedfunc$ is $\lipobj$-Lipschitz with respect to the $\ell_2$
    norm
  \item $\smoothedfunc$ is continuously differentiable; moreover,
    its gradient is $\frac{\lipobj}{u}$-Lipschitz continuous with
    respect to the $\ell_2$-norm.
  \item Let $Z \sim \smoothingdist$. Then $\E[\nabla F(x + Z; \statsample)] =
    \nabla \smoothedfunc(x)$, and $\E[\ltwo{\nabla \smoothedfunc(x) -
      \nabla F(x + Z; \statsample)}^2] \le \lipobj^2$.
  \end{enumerate}
  In addition, there exists a function $f$ for which each of the bounds
  (i)--(iv) cannot be improved by more than a constant factor.
\end{lemma}

Our final lemma illustrates the sharpness of the bounds we have proved for
functions that are Lipschitz with respect to the $\ell_2$-norm. Specifically,
we show that at least for the normal and uniform distributions, it is
impossible to get more favorable tradeoffs between the uniform approximation
error of the smoothed function $\smoothedfunc$ and the Lipschitz
continuity of $\nabla \smoothedfunc$. We begin with the following
definition of our two types of error, then give the lemma:
\begin{align}
  \erroruniform(f)
  & \defeq
  \inf\big\{L \in \R \mid \sup_{x \in \xdomain} |f(x) - \smoothedfunc(x)|
  \le L\big\} \label{eqn:def-error-uniform} \\
  \errorgrad(f)
  & \defeq
  \inf\big\{L \in \R \mid \ltwo{\nabla \smoothedfunc(x)
    - \nabla \smoothedfunc(y)} \le L \ltwo{x - y}
  ~ \forall ~ x, y \in \xdomain\big\}
  \label{eqn:def-error-grad}
\end{align}
\begin{lemma}
  \label{lemma:smoothing-sharp}
  For $\smoothingdist$ equal to either the uniform distribution on $B_2(0,
  \smoothparam)$ or $N(0, u^2 I_{d \times d})$, there exists an
  $\lipobj$-Lipschitz continuous function $f$ and dimension independent
  constant $c > 0$ such that
  \begin{equation*}
    \erroruniform(f) \errorgrad(f) \ge c \lipobj^2 \sqrt{d}.
  \end{equation*}
\end{lemma}
\begin{remark}
  The importance of the above bound is made clear by inspecting the
  convergence guarantee of Theorem~\ref{theorem:main-theorem-no-horizon}. The
  terms $\lipgrad$ and $\lipobj$ in the bound~\eqref{EqnNoHorizon} can be
  replaced with $\errorgrad(f)$ and $\erroruniform(f)$, respectively.
  Minimizing over $\smoothparam$, we see that the leading term in the
  convergence guarantee~\eqref{EqnNoHorizon} is of order
  $\frac{\sqrt{\errorgrad(f) \erroruniform(f) \prox(x^*)}}{T}
  \ge \frac{c\lipobj d^{1/4} \sqrt{\prox(x^*)}}{T}$. In
  particular, this result shows that our analysis of the dimension dependence
  of the randomized smoothing in Lemmas~\ref{lemma:ltwo}
  and~\ref{lemma:normal} is sharp and cannot be improved by
  more than a constant factor
  (see also Corollaries~\ref{corollary:variance-rate-ltwo}
  and~\ref{corollary:variance-rate-normal}).
\end{remark}

\subsection{Proof of smoothing lemmas}

The following technical lemma is a building block for our results;
we provide a proof in Sec.~\ref{sec:proof-of-lemma-lipschitz-gradient}.
\begin{lemma}
  \label{lemma:lipschitz-gradient}
  Let $f$ be convex and $\lipobj$-Lipschitz continuous with respect to a norm
  $\norm{\cdot}$ over the domain $\supp \smoothingdist + \xdomain$. Let
  $Z$ be distributed according to the distribution $\smoothingdist$. Then
  \begin{equation}
    \label{eqn:lipschitz-grad-bound}
    \dnorm{\nabla \smoothedfunc(x) - \nabla \smoothedfunc(y)}
    = \dnorm{\E\left[\nabla f(x + Z) + \nabla f(y + Z)\right]}
    \le \lipobj \int |\smoothingdist(z - x) - \smoothingdist(z - y))| dz.
  \end{equation}
  If the norm $\norm{\cdot}$ is the $\ell_2$-norm and the density
  $\smoothingdist(z)$ is rotationally symmetric and non-increasing
  as a function of $\ltwo{z}$, the
  bound~\eqref{eqn:lipschitz-grad-bound} is tight; specifically, it is
  attained by the function
  \begin{equation*}
    f(x) = \lipobj\left|\<\frac{y}{\ltwo{y}}, x\> - \half\right|.
  \end{equation*}
\end{lemma}

\subsubsection{Proof of Lemma~\ref{lemma:linf}}

Throughout, we let $Z \sim \smoothingdist$, where $\smoothingdist$ is
the uniform density on $B_\infty(0, u))$, and $h_\smoothparam(x)$
denote the (shifted) Huber loss
\begin{equation}
\label{EqnHuber}
h_\smoothparam(x) = \left\{\begin{array}{ll} \frac{x^2}{2u} +
\frac{u}{2} & {\rm for~} x \in [-u, u] \\ |x| & {\rm
  otherwise.} \end{array}\right.
\end{equation}
Now we prove each of the parts of the lemma in turn.
\begin{enumerate}[(i)]
\item Since $\E[Z] = 0$, Jensen's inequality shows $f(x) = f(x + \E[Z]) \le
  \E[f(x + Z)] = \smoothedfunc(x)$, by definition of
  $\smoothedfunc$. Now recall the definition of $\norm{\partial f(x)} =
  \sup\{\norm{g} \mid g \in \partial f(x)\}$ from the introduction. To get the
  upper uniform bound, note first that by assumption, $f$ is
  $\lipobj$-Lipschitz continuous over $\xdomain + B_\infty(0, u)$, since by
  assumption
  \begin{equation*}
    \linf{\partial f(x)} \le \E[\linf{\partial F(x; \statsample)}] \le
    \sqrt{\E[\linf{\partial F(x; \statsample)}^2]} \le \lipobj,
  \end{equation*}
  again
  using Jensen's inequality. Thus $f$ is $\lipobj$-Lipschitz with
respect to the $\ell_1$-norm,
\begin{equation*} \smoothedfunc(x) =
\E[f(x + Z)] \le \E[f(x)] + \lipobj\E[\lone{Z}] = f(x) +
\frac{d\lipobj u}{2}.
 \end{equation*}
 To see that the estimate is tight, note that for $f(x) = \lone{x}$,
 we have $\smoothedfunc(x) = \sum_{i=1}^d h_\smoothparam(x_i)$,
 where $h_\smoothparam$ is the shifted Huber loss~\eqref{EqnHuber},
 and $\smoothedfunc(0) = du/2$, while $f(0) = 0$.
\item We now prove that $\smoothedfunc$ is $\lipobj$-Lipschitz with
  respect to $\lone{\cdot}$.  Under the stated conditions, we have
  $\partial f(x) = \E[ \partial F(x; \statsample)]$, which shows that
  $\linf{\partial f(x)}^2 \le \E[\linf{\partial F(x; \statsample)}^2]
  \le \lipobj^2$.  Thus, we obtain the upper bound
 \begin{equation*}
  \linf{\nabla \smoothedfunc(x)} = \linf{\E[\nabla f(x + Z)]} \le
  \E[\linf{\nabla f(x + Z)}] \le \lipobj.
 \end{equation*}
  Tightness follows again by considering $f(x) = \lone{x}$, where
  $\lipobj = 1$.
\item Recall that differentiability is directly implied by earlier work of
  Bertsekas~\cite{Bertsekas73}. Since $f$ is a.e.-differentiable, we have
  $\nabla \smoothedfunc(x) = \E[\nabla f(x + Z)]$ for $Z$ uniform on
  $[-\smoothparam, \smoothparam]^d$. We now establish Lipschitz continuity of
  $\nabla \smoothedfunc(x)$.

  For a fixed pair $x, y \in \xdomain + B_\infty(0, \smoothparam)$, we
  have from Lemma~\ref{lemma:lipschitz-gradient}
  \begin{align*}
    \linf{\E[\nabla f(x + Z)] - \E[\nabla f(y + Z)]}
    & \le \lipobj \cdot \frac{1}{(2\smoothparam)^d}
    \lebesgue \big(B_\infty(x, u) \setdiff B_\infty(y, u)\big),
  \end{align*}
  where $\lebesgue$ denotes Lebesgue measure and $\setdiff$ denotes
  the symmetric set-difference. By a straightforward geometric
  calculation, we see that
  \begin{equation}
    \label{eqn:measure-infinity-difference}
    \lebesgue\left(B_\infty(x, u) \setdiff B_\infty(y, u)\right)
    = 2 \bigg((2u)^d - \prod_{i=1}^d \hinge{2u - |x_i - y_i|}\bigg).
  \end{equation}
  To control the volume term~\eqref{eqn:measure-infinity-difference} and
  complete the proof, we need an auxiliary lemma (which we prove at the end of
  this subsection).
  \begin{lemma}
    \label{lemma:linf-optimization}
    Let $a \in \R^d_+$ and $u \in \R_+$. Then $\prod_{i=1}^d \hinge{u - a_i}
    \ge u^d - \lone{a} u^{d-1}$.
  \end{lemma}
  The volume~\eqref{eqn:measure-infinity-difference} is
  easy to control using Lemma~\ref{lemma:linf-optimization}. Indeed, we have
  \begin{equation*}
    \half \lebesgue\left(B_\infty(x, u) \setdiff B_\infty(y, u)\right)
    \le (2u)^d - (2u)^d + \lone{x - y} (2u)^{d-1},
  \end{equation*}
  which implies the desired result, that is, that
  \begin{equation*}
    \linf{\E[\nabla f(x + Z)] - \E \nabla[f(y + Z)]}
    \le \frac{\lipobj \lone{x - y}}{u}.
  \end{equation*}
  To see the tightness claimed in the proposition, consider as usual $f(x) =
  \lone{x}$ and let $e_i$ denote the $i$th standard basis vector. Then
  $\lipobj = 1$, $\nabla \smoothedfunc(0) = 0$, $\nabla
  \smoothedfunc(u e_i) = e_i$, and $\linf{\nabla \smoothedfunc(0) -
    \nabla \smoothedfunc(u e_i)} = 1 = \frac{\lipobj}{u} \lone{0 - u
    e_i}$.
\item
  The equality $\E[\nabla F(x + Z; \statsample)] = \nabla
  \smoothedfunc(x)$ follows from Fubini's theorem. The second
  statement is simply a consequence of the triangle inequality.
  Finally, the tightness follows from the following one-dimensional
  example. Let $f(x) = \lipobj|x|$ for $x \in \R$ and $\lipobj >
  0$. Then $\smoothedfunc(x)$ is $\lipobj$ times the Huber loss
  $h_\smoothparam(x)$ defined earlier, and $f'_\smoothingdist(0) =
  0$. Thus for $Z$ uniform on $[-\smoothparam, \smoothparam]$,
  \begin{equation*}
    \E (\smoothedfunc'(0) - f'(Z))^2
    = \E[\lipobj^2 \sign(Z)^2] = \lipobj^2,
  \end{equation*}
  which is the Lipschitz constant of $f$.
\end{enumerate}

\begin{proof-of-lemma}[\ref{lemma:linf-optimization}]
  We begin by noting that the statement of the lemma trivially holds
  whenever $\lone{a} \ge u$, as the right hand side of the inequality
  is then non-positive. Now, fix some $c < u$, and consider the
  problem
  \begin{equation}
    \min_a ~ \prod_{i=1}^d (u - a_i)_+
    ~~~ {\rm s.t.} ~~~
    a \succeq 0, \lone{a} \le c.
    \label{eqn:linf-optimization}
  \end{equation}
  We show that the minimum is achieved when one index is set to $a_i =
  c$ and the rest to $0$. Indeed, suppose for the sake of contradiction
  that $\tilde{a}$ is the solution to \eqref{eqn:linf-optimization}
  but that there are indices $i, j$ with $a_i \ge a_j > 0$, that is,
  at least two non-zero indices.
  By taking a logarithm, it is clear that minimizing
  the objective~\eqref{eqn:linf-optimization} is equivalent to
  minimizing $\sum_{i=1}^d \log(u - a_i)$. Taking the derivative of
  $\log(u - a_i)$ for $i$ and $j$, we see that
  \begin{equation*}
    \frac{\partial}{\partial a_i} \log(u - a_i)
    = \frac{-1}{u - a_i}
    \le \frac{-1}{u - a_j} = \frac{\partial}{\partial a_j} \log(u - a_j).
  \end{equation*}
  Since $\frac{-1}{u - a}$ is decreasing function of $a$, increasing $a_i$
  slightly and decreasing $a_j$ slightly causes $\log(u - a_i)$ to decrease
  faster than $\log(u - a_j)$ increases, thus decreasing the overall
  objective. This is the desired contradiction.
%
\end{proof-of-lemma}


\subsubsection{Proof of Lemma~\ref{lemma:linf-to-ltwo}}

The proof of this lemma is nearly identical to the proof of
Lemma~\ref{lemma:linf}, though we replace $\linf{\cdot}$ norms with
$\ltwo{\cdot}$. We prove each of the statements in turn, and
throughout let $Z$ denote a variable distributed uniformly on
$B_\infty(0, \smoothparam)$.
\begin{enumerate}[(i)]
\item Jensen's inequality implies that $f(x) = f(x + \E[Z]) \le \E[f(x
  + Z)] = \smoothedfunc(x)$. For the upper bound on
  $\smoothedfunc$, use the Lipschitz continuity of $f$ and Jensen's
  inequality to see that
    \begin{equation*}
      \smoothedfunc(x) \le f(x) + \lipobj \E[\ltwo{Z}] \le f(x) +
      \lipobj \sqrt{\E[\ltwo{Z}^2]} = f(x) + \lipobj \sqrt{\frac{d
          \smoothparam^2}{3}}.
    \end{equation*}
  \item As earlier, since $\E[\nabla f(x + Z)] = \nabla \smoothedfunc(x)$,
    we have $\ltwo{\E[\nabla f(x + Z)]} \le \E[\ltwo{\nabla f(x + Z)}] \le
    \lipobj$.
  \item Using the same sequence of steps as in the proof of part (iii) in
    Lemma~\ref{lemma:linf}, we see that
    \begin{align*}
      \ltwo{\nabla \smoothedfunc(x) - \nabla \smoothedfunc(y)}
      & \le \frac{1}{(2\smoothparam)^d}
      \lipobj \lebesgue\left(B_\infty(x, \smoothparam)
      \setdiff B_\infty(y, \smoothparam)\right) \\
      & \le \frac{2}{(2\smoothparam)^d} \lipobj (2\smoothparam)^{d-1}
      \lone{x - y}
      \le \frac{\lipobj \sqrt{d}}{\smoothparam} \ltwo{x - y}.
    \end{align*}
  \item As in the proof of Lemma~\ref{lemma:linf}, Fubini's theorem
    implies the first part of the statement, while the second part
    is a consequence of the fact that
    \begin{equation*}
      \E[\ltwo{\nabla \smoothedfunc(x) - \nabla F(x + Z; \statsample)}^2]
      = \E[\ltwo{\nabla F(x + Z; \statsample)}^2] -
      \ltwo{\nabla \smoothedfunc(x)}^2
      \le \lipobj^2
    \end{equation*}
    by the assumptions on $F$. Tightness follows from considering the
    one dimensional function $f(x) = |x|$ as earlier.
  \end{enumerate}

\subsubsection{Proof of Lemma~\ref{lemma:normal}}
\newcommand{\normalintegral}{I_2}

Throughout this proof, we use $Z$ to denote a random variable
distributed as $N(0, \smoothparam^2 I)$.
\begin{enumerate}[(i)]
\item As in the earlier lemmas, Jensen's inequality gives $f(x) = f(x + \E
  Z) \le \E f(x + Z) = \smoothedfunc(x)$. Our assumption on $\partial
  F(\cdot; \statsample)$ implies that $f$ is $\lipobj$-Lipschitz, so
  \begin{equation*}
    \smoothedfunc(x) = \E[f(x + Z)]
    \le \E[f(x)] + \lipobj \E[\ltwo{Z}]
    \le f(x) + \lipobj \sqrt{\E[\ltwo{Z}^2]}
    = f(x) + \lipobj \smoothparam \sqrt{d}.
  \end{equation*}
\item This proof is analogous to that of part (ii) of
  Lemmas~\ref{lemma:linf} and~\ref{lemma:linf-to-ltwo}.  The
  tightness of the Lipschitz constant can be verified by taking
  $f(x) = \<v, x\>$ for $v \in \R^d$, in which case
  $\smoothedfunc(x) = f(x)$, and both have gradient $v$.
\item Now we show that $\nabla \smoothedfunc$ is Lipschitz continuous.
  Indeed, applying Lemma~\ref{lemma:lipschitz-gradient} we have
  \begin{align}
    \ltwo{\nabla \smoothedfunc(x) - \nabla \smoothedfunc(y)}
    & \le \lipobj \underbrace{\int |\smoothingdist(z - x)
      - \smoothingdist(z - y)| dz.}_{\normalintegral}
    \label{eqn:difference-probability-measures}
  \end{align}
  What remains is to control the integral
  term~\eqref{eqn:difference-probability-measures}, denoted $\normalintegral$.

  In order to do so, we follow a technique used by Lakshmanan and Pucci de
  Farias~\cite{LakshmananFa08}.  Since $\smoothingdist$ satisfies
  $\smoothingdist(z - x) \ge \smoothingdist(z - y)$ if and only if $\ltwo{z -
    x} \ge \ltwo{z - y}$, we have
  \begin{align*}
    \normalintegral & = \int|\smoothingdist(z - x) - \smoothingdist(z - y)| dz
    = 2 \int_{z : \ltwo{z - x} \le \ltwo{z - y}} (\smoothingdist(z - x)
    - \smoothingdist(z - y)) dz.
  \end{align*}
  By making the change of variable $w = z - x$ for the $\smoothingdist(z - x)$
  term in $\normalintegral$ and $w = z - y$ for $\smoothingdist(z - y)$,
  we rewrite $\normalintegral$ as
  \begin{align*}
    \normalintegral & = 2 \int_{w : \ltwo{w} \le \ltwo{w -
        (x - y)}} \smoothingdist(w) dw - 2 \int_{w : \ltwo{w} \ge \ltwo{w
        - (x - y)}} \smoothingdist(w) dw \\
    & = 2\P_\smoothingdist(\ltwo{Z} \le \ltwo{Z - (x - y)}) -
    2\P_\smoothingdist(\ltwo{Z} \ge \ltwo{Z - (x - y)})
  \end{align*}
  where $\P_\smoothingdist$ denotes probability according to the density
  $\smoothingdist$. Squaring the terms inside the probability bounds, we note
  that
  \begin{align*}
    \P_\smoothingdist\left(\ltwo{Z}^2 \le \ltwo{Z - (x - y)}^2\right)
    & = \P_\smoothingdist\left(2 \<Z, x - y\> \le \ltwo{x - y}^2\right) \\
    & = \P_\smoothingdist\left(2 \<Z, \frac{x - y}{\ltwo{x - y}}\>
    \le \ltwo{x - y}\right)
  \end{align*}
  
  Since $(x - y) / \ltwo{x - y}$ has norm 1 and $Z \sim N(0, \smoothparam^2
  I)$ is rotationally invariant, the random variable $W = \<Z, \frac{x -
    y}{\ltwo{x - y}}\>$ has distribution $N(0, \smoothparam^2)$.
  Consequently, we have
  \begin{align*}
    \frac{\normalintegral}{2}
    & =  \lefteqn{\P\left(W \le \ltwo{x - y}/2\right)
      - \P \left(W \ge \ltwo{x - y}/2 \right)} \\
    & = \int_{-\infty}^{\ltwo{x - y}/2} \frac{1}{\sqrt{2\pi \smoothparam^2}}
    \exp(-w^2 / (2\smoothparam^2)) dw - \int_{\ltwo{x - y}/2}^\infty
    \frac{1}{\sqrt{2\pi \smoothparam^2}} \exp(-w^2 / (2\smoothparam^2)) dw \\
    & \le \frac{1}{\smoothparam \sqrt{2\pi}} \ltwo{x - y},
  \end{align*}
  where we have exploited symmetry and the inequality $\exp(-w^2) \le
  1$. Combining this bound with the earlier
  inequality~\eqref{eqn:difference-probability-measures}, we have
  \begin{equation*}
    \ltwo{\nabla \smoothedfunc(x) - \nabla \smoothedfunc(y)}
    \le \frac{2 \lipobj}{\smoothparam \sqrt{2 \pi}} \ltwo{x - y}
    \le \frac{\lipobj}{\smoothparam} \ltwo{x - y}.
  \end{equation*}
\item The proof of the variance bound is completely identical to that for
  Lemma~\ref{lemma:linf-to-ltwo}.
\end{enumerate}

That each of the bounds above is tight is a consequence of
Lemma~\ref{lemma:smoothing-sharp}.

\subsubsection{Proof of Lemma~\ref{lemma:smoothing-sharp}}

Throughout this proof, $c$ will denote a dimension independent constant and
may change from line to line and inequality to inequality. We will show the
result holds by considering a convex combination of ``difficult'' functions,
in this case $f_1(x) = \lipobj \ltwo{x}$ and $f_2(x) = \lipobj \left|\<x, y /
\ltwo{y}\> - 1/2\right|$, and choosing $f = \half f_1 + \half f_2$.
Our first step in the proof will be to control $\erroruniform$.

By definition of the constant $\erroruniform$ in
Eq.~\eqref{eqn:def-error-uniform}, for any convex $f_1$ and $f_2$ we have
$\erroruniform(\half f_1 + \half f_2) \ge \half \max\{\erroruniform(f_1),
\erroruniform(f_2)\}$. Thus for $Z \sim N(0, \smoothparam^2 I_{d \times d})$
we have $\E[f_1(Z)] \ge c\lipobj \smoothparam \sqrt{d}$,
i.e.\ $\erroruniform(f) \ge c \lipobj \smoothparam \sqrt{d}$, and for $Z$
uniform on $B_2(0, \smoothparam)$, we have $\E[f_1(Z)] \ge c \lipobj
\smoothparam$, i.e.\ $\erroruniform(f) \ge c \lipobj \smoothparam$.

Turning to control of $\errorgrad$, we note that for any random variable $Z$
rotationally symmetric about the origin, symmetry implies that
\begin{equation*}
  \E[\nabla f_1(Z + y)] = \lipobj \E\left[\frac{Z + y}{\ltwo{Z + y}}\right]
  = a_z y
\end{equation*}
where $a_z > 0$ is a constant dependent on $Z$.
Thus we have
\begin{equation*}
  \E[\nabla f_1(Z)] - \E[\nabla f_1(Z + y)]
  + \E[\nabla f_2(Z)] - \E[\nabla f_2(Z + y)]
  = 0 - a_z y - \lipobj \frac{y}{\ltwo{y}} \int|\smoothingdist(z)
  - \smoothingdist(z - y)| dz
\end{equation*}
from Lemma~\ref{lemma:lipschitz-gradient}. As a consequence (since $a_z y$
is parallel to $y/\ltwo{y}$), we see that
\begin{equation*}
  \errorgrad\left(\half f_1 + \half f_2\right)
  \ge \half \lipobj \int|\smoothingdist(z) - \smoothingdist(z - y)| dz.
\end{equation*}
So what remains is to lower bound $\int|\smoothingdist(z) - \smoothingdist(z
- y)| dz$ for the uniform and normal distributions. As we saw in the proof
of Lemma~\ref{lemma:normal},
for the normal distribution
\begin{equation*}
  \int|\smoothingdist(z) - \smoothingdist(z - y)| dz
  = \frac{1}{\smoothparam \sqrt{2\pi}} \int_{-\ltwo{y}/2}^{\ltwo{y}/2}
  \exp(-w^2 / (2\smoothparam^2)) dw
  = \frac{1}{\smoothparam \sqrt{2\pi}} \ltwo{y} + \order\left(
  \frac{\ltwo{y}^2}{\smoothparam}\right).
\end{equation*}
By taking small enough $\ltwo{y}$, we achieve the inequality
$\errorgrad\!\left(\half f_1 + \half f_2\right)
\ge c \frac{\lipobj}{\smoothparam}$
when $Z \sim N(0, \smoothparam^2 I_{d \times d})$.

To show that the bound in the lemma is sharp for the case of the uniform
distribution on $B_2(0, \smoothparam)$, we slightly modify the proof of
Lemma~2 in~\cite{YousefianNeSh12}. In particular, by using a Taylor
expansion instead of first-order convexity in inequality~(11)
of~\cite{YousefianNeSh12}, it is not difficult to show that
\begin{equation*}
  \int|\smoothingdist(z) - \smoothingdist(z - y)| dz
  = \kappa \frac{d!!}{(d - 1)!!} \frac{\ltwo{y}}{\smoothparam}
  + \order\bigg(\frac{d \ltwo{y}^2}{\smoothparam^2}\bigg),
\end{equation*}
where $\kappa = 2/\pi$ if $d$ is even and $1$ otherwise.  Since $d!! / (d -
1)!! = \Theta(\sqrt{d})$, we have proved that for small enough $\ltwo{y}$,
there is a constant $c$ such that $\int|\smoothingdist(z) - \smoothingdist(z
- y)| dz \ge c \sqrt{d} \ltwo{y} / \smoothparam$.

\subsubsection{Proof of Lemma~\ref{lemma:lipschitz-gradient}}
\label{sec:proof-of-lemma-lipschitz-gradient}

Without loss of generality, we assume that $x = 0$ (a linear change of
variables allows this).  Let $g : \R^d \rightarrow \R^d$ be a vector-valued
function such that $\dnorm{g(z)} \le \lipobj$ for all $z \in \{y\} + \supp
\smoothingdist$. Then
\begin{align}
  \E[g(Z) - g(y + Z)]
  & = \int g(z) \smoothingdist(z) dz - \int g(y + z) \smoothingdist(z) dz
  \nonumber \\
  & = \int g(z) \smoothingdist(z) dz - \int g(z) \smoothingdist(z - y) dz
  \nonumber \\
  & = \int_{\gtset} g(z)[\smoothingdist(z) - \smoothingdist(z - y)] dz
  - \int_{\ltset} g(z)[\smoothingdist(z - y) - \smoothingdist(z)] dz
  \label{eqn:split-to-gt-lt}
\end{align}
where $\gtset = \{z \in \R^d \mid \smoothingdist(z) > \smoothingdist(z -
y)\}$ and $\ltset = \{z \in \R^d \mid \smoothingdist(z) < \smoothingdist(z -
y)\}$. It is now clear that when we take norms we have
\begin{align*}
  \dnorm{\E g(Z) - g(y + Z)}
  & \le \sup_{z \in \gtset \cup \ltset} \dnorm{g(z)}
  \left|\int_{\gtset}[\smoothparam(z) - \smoothparam(z - y)]dz
  + \int_{\ltset} [\smoothparam(z - y) - \smoothparam(z)] dz\right| \\
  & \le \lipobj \left|\int_{\gtset} \smoothingdist(z) - \smoothingdist(z - y)
  dz
  + \int_{\ltset}\smoothingdist(z - y) - \smoothingdist(z) dz \right| \\
  & = \lipobj \int |\smoothingdist(z) - \smoothingdist(z - y)| dz.
\end{align*}
Taking $g(z)$ to be an arbitrary element of $\partial f(z)$ completes the
proof of the bound~\eqref{eqn:lipschitz-grad-bound}.

To see that the result is tight when $\smoothingdist$ is rotationally
symmetric and the norm $\norm{\cdot} = \ltwo{\cdot}$, we note the
following. From the equality~\eqref{eqn:split-to-gt-lt}, we see that
$\ltwo{\E[g(Z) - g(y + Z)]}$ is maximized by choosing $g(z) = v$ for $z \in
\gtset$ and $g(z) = -v$ for $z \in \ltset$ for any $v$ such that $\ltwo{v} =
\lipobj$.  Since $\smoothingdist$ is rotationally symmetric and
non-increasing in $\ltwo{z}$,
\begin{align*}
  \gtset & = \left\{z \in \R^d \mid \smoothingdist(z) > \smoothingdist(z - y)
  \right\}
  = \left\{z \in \R^d \mid \ltwo{z}^2 < \ltwo{z - y}^2\right\}
  = \left\{z \in \R^d \mid \<z, y\> < \half \ltwo{y}^2\right\} \\
  \ltset & = \left\{z \in \R^d \mid \smoothingdist(z) < \smoothingdist(z - y)
  \right\}
  = \left\{z \in \R^d \mid \ltwo{z}^2 > \ltwo{z - y}^2\right\}
  = \left\{z \in \R^d \mid \<z, y\> > \half \ltwo{y}^2\right\}.
\end{align*}
So all we need do is find a function $f$ for which there exists $v$
with $\ltwo{v} = \lipobj$, and such that $\partial f(x) = \{v\}$ for $x \in
\gtset$ and $\partial f(x) = \{-v\}$ for $x \in \ltset$. By inspection, the
function $f$ defined in the statement of the lemma satisfies these two
desiderata for $v = \lipobj \frac{y}{\ltwo{y}}$.

\section{Sub-Gaussian and sub-exponential tail bounds}
\label{AppTail}

For reference purposes, we state here some standard definitions and
facts about sub-Gaussian and sub-exponential random variables (see the
books~\cite{BuldyginKo00,LedouxTa91,VanDerVaartWe96} for further
details).

\subsection{Sub-Gaussian variables}

This class of random variables is characterized by a quadratic upper
bound on the moment generating function:
\begin{definition}
  \label{definition:sub-Gaussian}
  A zero-mean random variable $X$ is called \emph{sub-Gaussian with parameter}
  $\sigma^2$ if $\E \exp(\lambda X) \le \exp(\sigma^2
  \lambda^2 / 2)$ for all $\lambda \in \R$.
\end{definition}
\paragraph{Remarks:}  If $X_i$, $i = 1,\ldots, n$ are independent 
sub-Gaussian with parameter $\sigma^2$, it follows from this
definition that $\frac{1}{n} \sum_{i=1}^n X_i$ is sub-Gaussian with
parameter $\sigma^2/ n$.
%
Moreover, it is well-known that any zero-mean random variable $X$
satisfying $|X| \le C$ is sub-Gaussian with parameter $\sigma^2 \le
C^2$. \\

\begin{lemma}[Buldygin and Kozachenko~\cite{BuldyginKo00}, Lemma 1.6]
  \label{lemma:second-order-subg}
  Let $X - \E X$ be sub-Gaussian with parameter $\sigma^2$. Then
  for $s \in [0, 1]$,
  \begin{equation*}
    \E \exp\left(\frac{s X^2}{2 \sigma^2}\right)
    \le \frac{1}{\sqrt{1 - s}} \exp\left(\frac{(\E X)^2}{2 \sigma^2}
    \cdot \frac{s}{1 - s}\right).
  \end{equation*}
\end{lemma}

\vspace*{.1in}

The maximum of $d$ sub-Gaussian random variables grows logarithmically
in $d$, as shown by the following result:
\begin{lemma}
  \label{lemma:sub-gaussian}
  Let $X \in \R^d$ be a random vector with sub-Gaussian components, each with
  parameter at most $\sigma^2$. Then $\E \linf{X}^2 \le
  \max\{6 \sigma^2 \log d, 2 \sigma^2\}$.
\end{lemma}
\noindent
Using the definition of sub-Gaussianity, the result can be proved by
a combination of union bounds and Chernoff's inequality (see van der
Vaart and Wellner~\cite[Lemma 2.2.2]{VanDerVaartWe96} or Buldygin and
Kozachenko~\cite[Chapter II]{BuldyginKo00} for details). \\

\noindent The following martingale-based bound for variables with
conditionally sub-Gaussian behavior is essentially
standard~\cite{Azuma67,Hoeffding63,BuldyginKo00}.
\begin{lemma}[Azuma-Hoeffding]
  Let $X_i$ be a martingale difference sequence adapted to the
  filtration $\mc{F}_i$, and assume that each $X_i$ is conditionally
  sub-Gaussian with parameter $\sigma_i^2$, meaning that
  $\E[\exp(\lambda X_i) \mid \mc{F}_{i-1}] \le \exp(\lambda^2 \sigma_i^2
  / 2)$. Then for all $\epsilon > 0$,
  \begin{equation}
    \P \bigg[\ninv \sum_{i=1}^n X_i \ge \epsilon \biggr] \le \exp
    \biggr(-\frac{n \epsilon^2}{2 \sum_{i=1}^n \sigma_i^2/n} \biggr).
    \label{eqn:hoeffding-azuma}
  \end{equation}
\end{lemma}

The next lemma uses martingale techniques to establish
the sub-Gaussianity of a normed sum:
\begin{lemma}
  \label{lemma:subgaussian-bounded-vectors}
  Let $X_1, \ldots, X_n$ be independent random vectors with
  $\norm{X_i} \le L$ for all $i$. Define $S_n = \sum_{i=1}^n
  X_i$. Then $\norm{S_n} - \E \norm{S_n}$ is sub-Gaussian with
  parameter at most $4nL^2$.
\end{lemma}
\begin{proof}
  The proof follows from the realization that when $\norm{X_i} \le L$,
  the quantity $\norm{S_n} - \E\norm{S_n}$ can be controlled using
  single-dimensional martingale techniques~\cite[Chapter
  6]{LedouxTa91}. We construct the Doob martingale for the sequence
  $X_i$.  Let $\mc{F}_i$ be the $\sigma$-field of $X_1, \ldots, X_i$
  and define the real-valued random variables $Z_i = \E[\norm{S_n}
  \mid \mc{F}_i] - \E[\norm{S_n} \mid \mc{F}_{i-1}]$, where $\mc{F}_0$
  is the trivial $\sigma$-field. Let $S_{n\setminus i} = \sum_{j \neq
  i} X_j$. Then $\E[Z_i \mid \mc{F}_{i-1}] = 0$ and
  \begin{align*}
    |Z_i| & =
    \left|\E[\norm{S_n} \mid \mc{F}_{i-1}]
    - \E[\norm{S_n} \mid \mc{F}_i]\right| \\
    & \le \left|\E\left[\norm{S_{n \setminus i}} \mid \mc{F}_{i-1}\right]
    - \E\left[\norm{S_{n \setminus i}} \mid \mc{F}_i\right]\right|
    + \E[\norm{X_i} \mid \mc{F}_{i-1}]
    + \E[\norm{X_i} \mid \mc{F}_i] \\
    & = \norm{X_i} + \E[\norm{X_i}] \le 2L
  \end{align*}
  since $X_j$ is independent of $\mc{F}_{i-1}$ for $j \ge i$. Thus
  $Z_i$ defines a bounded martingale difference sequence, and
  $\sum_{i=1}^n Z_i = \norm{S_n} - \E[\norm{S_n}]$. Since $|Z_i| \le
  2L$, the $Z_i$ are conditionally sub-Gaussian with parameter at most
  $4L^2$. Thus $\sum_{i=1}^n Z_i$ is sub-Gaussian with parameter at
  most $4nL^2$.
%
\end{proof}

\subsection{Sub-exponential random variables}
A slightly less restrictive tail condition defines the class
of sub-exponential random variables:
\begin{definition}
  \label{def:subexp}
  A zero-mean random variable $X$ is \emph{sub-exponential} with
 parameters $(\subexpbound, \subexpparam)$ if
  \begin{equation*}
    \E[\exp(\lambda X)] \le \exp\left(\frac{\lambda^2 \subexpparam^2}{2}\right)
    \quad \mbox{for all}~ |\lambda| \le \subexpbound.
  \end{equation*}
\end{definition}

\noindent The following lemma provides an equivalent characterization
of sub-exponential variable via a tail bound:
\begin{lemma}
  \label{lemma:concentration-to-subexp}
  Let $X$ be a zero-mean random variable. If there are constants
  $\subexpprob, \subexpexp > 0$ such that
  \begin{equation*}
    \P(|X| \ge t) \le \subexpprob \exp(-\subexpexp t)
    \quad \mbox{for all}~ t > 0
  \end{equation*}
  then $X$ is sub-exponential with parameters $\subexpbound =
  \subexpexp / 2$ and $\subexpparam^2 = 4\subexpprob / \subexpexp^2$.
\end{lemma}
\noindent
The proof of the lemma follows from a Taylor expansion of $\exp(\cdot)$ and
the identity $\E[|X|^k] = \int_0^\infty \P(|X|^k \ge t) dt$ (for similar
results, see Buldygin and Kozachenko~\cite[Chapter I.3]{BuldyginKo00}).

Lastly, any random variable whose square is sub-exponential is sub-Gaussian,
as shown by the following result:
\begin{lemma}[Lan, Nemirovski, Shapiro~\cite{LanNeSh10}, Lemma 6]
  \label{lemma:subexp-squared-implies-subg}
  Let $X$ be a zero-mean random variable satisfying the moment generating
  inequality $\E[\exp(X^2 / \sigma^2)] \le \exp(1)$. Then
  $X$ is sub-Gaussian with parameter at most $3/2 \sigma^2$.
\end{lemma}
%
\comment{
\begin{proof-of-lemma}[\ref{lemma:subexp-squared-implies-subg}]
  The proof uses the Fenchel-Young inequality, that is, that for any $a, b$
  and $c > 0$, $ab \le \frac{a^2}{2c} + \frac{cb^2}{2}$. We also note that by
  Jensen's inequality, for any $\lambda \in [0, 1]$, $\E[\exp(\lambda X^2 /
    \sigma^2)] \le \exp(\lambda)$. With these facts, we see that
  \begin{equation}
    \E[\exp(\lambda X)] \le \E\left[\exp\left(\frac{\lambda^2 \sigma^2}{2c}
      + \frac{c X^2}{2 \sigma^2}\right)\right]
    \le \exp\left(\frac{\lambda^2 \sigma^2}{2c} + \frac{c}{2}\right)
    \label{eqn:mgf-big-lambda}
  \end{equation}
  for any $\lambda$ and $c \in [0, 2]$. Now we use the fact that
  $\exp(x) \le x + \exp(9x^2 / 16)$ to see that
  \begin{equation}
    \E[\exp(\lambda X)] \le \E[\lambda X] + \E \biggr[\exp \big(\frac{
        9\lambda^2 X^2}{16} \big) \biggr] \le 0 + \exp
    \big(\frac{9\lambda^2 \sigma^2}{16} \big)
    \label{eqn:mgf-small-lambda}
  \end{equation}
  for $\lambda \in [0, 4/3\sigma]$. Recalling the
  bound~\eqref{eqn:mgf-big-lambda}, we see that if we take $c = 4/3$,
  then for $\lambda \ge 4 / 3 \sigma$,
  \begin{equation*}
    \frac{\lambda^2 \sigma}{2c}
    = \frac{3\lambda^2 \sigma^2}{8} \ge \frac{2}{3}
    = \frac{c}{2}.
  \end{equation*}
  In particular, $\frac{\lambda^2 \sigma^2}{2c} + \frac{c}{2} \le \frac{3
    \lambda^2 \sigma^2}{4}$ for $\lambda \ge 4/3\sigma$; combining this with
  the bound~\eqref{eqn:mgf-small-lambda} for small $\lambda$ gives the result.
\end{proof-of-lemma}
} 
\vspace*{.2in}

\bibliographystyle{abbrv}
\bibliography{bib}

\end{document}